\documentclass{article}

\title{Robust finite element discretization and solvers for 
  distributed elliptic optimal control problems}

\author{Ulrich~Langer\footnote{Institute of Computational Mathematics, 
    Johannes Kepler University Linz, Altenberger Stra{\ss}e 69, 4040 Linz, Austria, 
    Email: ulanger@numa.uni-linz.ac.at}, 
  \;
  Richard~L\"oscher\footnote{Institut f\"{u}r Angewandte Mathematik,
    Technische Universit\"{a}t Graz, Steyrergasse 30, 8010 Graz, Austria,
    Email: loescher@math.tugraz.at}, 
  \; Olaf~Steinbach\footnote{Institut f\"{u}r Angewandte Mathematik,
    Technische Universit\"{a}t Graz, Steyrergasse 30, 8010 Graz, Austria,
    Email: o.steinbach@tugraz.at}, 
  \; Huidong~Yang\footnote{Computational Science Center, Universit\"at
    Wien, Oskar--Morgenstern--Platz 1, 1090 Wien, Austria,
    Email: huidong.yang@univie.ac.at}}  

\date{\today}

\usepackage{a4}
\usepackage{paralist}
\usepackage{multirow}

\usepackage{amsmath}
\usepackage{amsthm}
\usepackage{amsfonts}
\usepackage{booktabs}
\usepackage{graphicx}
\usepackage{diagbox}

\usepackage{multirow}
\usepackage{hyperref}

\usepackage[ulem=normalem,draft]{changes}

\newtheorem{theorem}{Theorem}
\newtheorem{lemma}{Lemma}
\newtheorem{cor}{Corollary}

\numberwithin{equation}{section} 

\begin{document}

\maketitle

\begin{abstract}
We consider standard tracking-type, distributed elliptic optimal control 
problems with $L^2$ regularization, and their finite element discretization.
We are investigating the $L^2$ error between the finite element approximation 
$u_{\varrho h}$ of the state $u_\varrho$ and the desired state (target)
$\overline{u}$ in terms of the regularization parameter $\varrho$ and the
mesh size $h$ that leads to the optimal choice $\varrho = h^4$. 
It turns out that, for this choice of the regularization parameter,
we can devise simple Jacobi-like preconditioned MINRES or Bramble-Pasciak CG 
methods that allow us to solve the reduced discrete optimality system 
in asymptotically optimal complexity with respect to the arithmetical
operations and memory demand. The theoretical results are confirmed 
by several benchmark problems with targets of various regularities
including discontinuous targets.
\end{abstract} 

\begin{keywords}
  Elliptic optimal control problems, $L^2$ regularization,
  finite element discretization,
  robust error estimates, robust solvers.
\end{keywords}

\begin{msc}
49J20,  
49M05,  
35J05,  
65M60,  
65M15,  
65N22   
\end{msc}

%
%
\section{Introduction}
\label{sec:Introduction}
Let us consider the optimal control problem: Find the optimal state
$u_\varrho$ and the optimal control $z_\varrho$ such that the cost functional 
\begin{equation}\label{minimization problem}
  {\mathcal{J}}(u_\varrho,z_\varrho) =
  \frac{1}{2} \int_\Omega [u_\varrho(x)-\overline{u}(x)]^2 \, dx +
  \frac{\varrho}{2} \int_\Omega [z_\varrho(x)]^2 \, dx
\end{equation}
is minimized subject to the elliptic boundary value problem
\begin{equation}\label{primal problem}
  - \Delta u_\varrho = z_\varrho \quad \mbox{in} \; \Omega,
  \quad \mbox{and} \quad 
  u_\varrho = 0 \quad \mbox{on} \; \partial\Omega,
\end{equation}
where $\varrho > 0$ denotes the regularization parameter, $\overline{u}$ is
the given desired state that is nothing but the target which we want to reach,
and $\Omega \subset \mathbb{R}^d$ ($d=1,2,3$) denotes the computational
domain that is assumed to be bounded with Lipschitz boundary
$\partial \Omega$.  
This optimal control problem has a unique solution 
$u_\varrho \in H_0^1(\Omega)$ and $z_\varrho \in L^2(\Omega)$,
where we use the usual notation for Sobolev and Lebesgue spaces; 
see, e.g., \cite{Lions:1971, Troeltzsch:2010}. This solution can be
determined by solving the optimality system consisting of the state
(primal) problem \eqref{primal problem}, the gradient equation
\begin{equation}\label{gradient equation}
p_\varrho + \varrho z_\varrho = 0 \quad \mbox{in} \; \Omega,
\end{equation}
and the adjoint problem
\begin{equation}\label{adjoint problem}
  - \Delta p_\varrho = u_\varrho - \overline{u} \quad \mbox{in} \; \Omega, \quad
  p_\varrho = 0 \quad \mbox{on} \; \partial \Omega
\end{equation}
for determining the co-state (adjoint) $p_\varrho$.
The control $z_\varrho$ can be eliminated by means of the gradient equation
\eqref{gradient equation} that results in the so-called reduced optimality
system which is the starting point for our analysis in
Section~\ref{sec:RegularizationErrorEstimates}.
At that point we mention that there are many publications on elliptic optimal 
control problems such as \eqref{minimization problem}--\eqref{primal problem}, 
their numerical solution, and their application to practical problems, 
where often additional constraints, e.g., box constraints imposed on
the control, are added. We refer to the monographs
\cite{Lions:1971, Troeltzsch:2010} for more detailed information on
optimal control problems of that kind.

In this paper, we are interested to derive regularization
  error estimates of the form
$\|u_\varrho - \overline{u}\|_{L^2(\Omega)}$ and 
$\|u_{\varrho h} - \overline{u}\|_{L^2(\Omega)}$ in terms of the relaxation
parameter $\varrho$ and the finite element mesh size $h$. Optimal
control problems without state or control constraints correspond
to the solution of inverse problems with Tikhonov regularization in a
Hilbert scale; see \cite{Natterer:1984} for related regularization
error estimates. In \cite{LLSY:NeumuellerSteinbach:2021a} we have
given regularization error estimates for
$\| u_\varrho - \overline{u} \|_{L^2(\Omega)}$ when considering both the
regularization in $L^2(\Omega)$, and in the energy space $H^{-1}(\Omega)$.
The order of the convergence rate in the relaxation parameter $\varrho$
does not only depend on the regularity of the given target, but also
on the choice of the regularization. In the case of the energy
regularization, 
we have
analyzed the finite element discretization of the reduced optimality
system
in \cite{LLSY:LangerSteinbachYang:2022a}.
When combining both error estimates, this results in the
choice $\varrho = h^2$ to ensure optimal convergence of the approximate
state $u_{\varrho h}$ to the target $\overline{u}$. This optimal choice
of the regularization parameter also allows the construction of
preconditioned iterative solution methods which are robust with respect
to the regularization parameter $\varrho=h^2$, and the finite element
mesh size $h$. In this work we will derive related regularization
and finite element error estimates in the case of the $L^2$ regularization
which will result in the optimal choice $\varrho=h^4$.
This optimal choice of $\varrho$ allows us to construct robust and,
at the same time, very efficient iterative solvers based on diagonal,
i.e. Jacobi-like preconditioners. The discretization of the reduced
optimality system leads to large-scale systems of finite element
equations with symmetric, but indefinite system matrices
for determining the nodal vectors for the finite element approximations 
to the optimal state $u_\varrho$ and optimal co-state (adjoint) $p_\varrho$.
Iterative solvers for such kind of saddle point systems are extensively
studied in the literature. We refer the reader to the survey paper
\cite{LLSY:BenziGolubLiesen:2005a}, the monograph \cite{Elman2005}, 
and the more recent papers \cite{LLSY:Notay:2014a} and
\cite{LLSY:Notay:2019a} for a  comprehensive overview on saddle point
solvers. Using an operator interpolation technique, Zulehner proposed a block-diagonal preconditioner 
for the symmetric and indefinite discrete reduced optimality system that 
is robust with respect to the regularization parameter $\varrho$
\cite{LSTY:Zulehner:2011a}. The diagonal blocks of the preconditioner 
are of the form $\mathbf{M}_h + \varrho^{1/2} \mathbf{K}_h$ and
$\varrho^{-1} (\mathbf{M}_h + \varrho^{1/2} \mathbf{K})$, where
$\mathbf{M}_h$ and $\mathbf{K}_h$ denote the mass and the stiffness
matrices, respectively. Replacing now $\mathbf{M}_h + \varrho^{1/2}
\mathbf{K}_h$ by some $\varrho$ - robust multigrid or multilevel preconditioner
as proposed in \cite{LLSY:OlahanskiiReusken:2000a}
and \cite{LLST:KrausWolfmayr:2014a}, this
finally leads to a robust and efficient preconditioner for  
the MINRES solver \cite{LLSY:PaigeSaunders:1975a}.
Surprisingly, for the optimal choice $\varrho=h^4$ of the
regularization parameter, the matrix
$\mathbf{M}_h + \varrho^{1/2} \mathbf{K}_h$ is spectrally equivalent
to the mass matrix $\mathbf{M}_h$, and, therefore, well-conditioned.
Now, replacing $\mathbf{M}_h + \varrho^{1/2} \mathbf{K}_h$ by some diagonal 
approximation of the mass matrix, we get a robust and really very cheap
preconditioner for MINRES. The same observation lead to robust and 
asymptotically optimal preconditioners for the
Bramble-Pasciak CG \cite{LLSY:BramblePasciak:1988a}.

The remainder of the paper is organized as follows. 
In Section~\ref{sec:RegularizationErrorEstimates},
we derive the $L^2$ error estimate between the 
exact state solution $u_\varrho$ of the optimal control problem
for fixed $\varrho$ and the desired state $\overline{u}$ 
in terms of the regularization parameter $\varrho$,
whereas, in Section~\ref{sec:FiniteElementErrorEstimates},
the same estimates are derived for a finite element approximation
$u_{\varrho h}$ to $\overline{u}$.
Section~\ref{sec:RobustSolvers} is devoted to the construction 
and analysis of fast and robust iterative solvers.
Section~\ref{sec:NumericalResults} presents and discusses 
numerical results for typical benchmark problems.
Finally, in Section~\ref{sec:ConclusionsOutlook}, 
we draw some conclusions, and give an outlook on 
some future research topics.

%
%

\section{Regularization error estimates}
\label{sec:RegularizationErrorEstimates}
When using the gradient equation \eqref{gradient equation} to eliminate the
control $z_\varrho$, the variational formulation of the primal Dirichlet
problem \eqref{primal problem} is to find $u_\varrho \in H^1_0(\Omega)$ such that
\begin{equation}\label{VF primal}
  \frac{1}{\varrho} \langle p_\varrho , v \rangle_{L^2(\Omega)} +
  \langle \nabla u_\varrho , \nabla v \rangle_{L^2(\Omega)} = 0 \quad
  \mbox{for all} \; v \in H^1_0(\Omega),
\end{equation}
while the variational formulation of the adjoint problem
\eqref{adjoint problem} is to find $p_\varrho \in H^1_0(\Omega)$ such that
\begin{equation}\label{VF adjoint}
  \langle \nabla p_\varrho , \nabla q \rangle_{L^2(Q)} =
  \langle u_\varrho - \overline{u} , q \rangle_{L^2(Q)} \quad \mbox{for all} \;
  q \in H^1_0(\Omega) .
\end{equation}
While unique solvability of the coupled variational formulation
\eqref{VF primal} and \eqref{VF adjoint} is well established, our
particular interest is in estimating the regularization error
$\| u_\varrho - \overline{u} \|_{L^2(\Omega)}$ which depends 
on the regularity of the target $\overline{u}$. As already discussed in
\cite{LLSY:NeumuellerSteinbach:2021a}, we can prove the following result:

\begin{lemma}\label{Lemma Regularization L2}
  Let $(u_\varrho,p_\varrho) \in H^1_0(\Omega) \times H^1_0(\Omega)$ be the
  unique solution of the coupled variational formulation
  \eqref{VF primal} and \eqref{VF adjoint}. When assuming
  $\overline{u} \in L^2(\Omega)$ only, this gives
  \begin{equation}\label{regularization L2 L2}
    \| u_\varrho - \overline{u} \|_{L^2(\Omega)}
    \leq \|\overline{u} \|_{L^2(\Omega)} .
  \end{equation}
  For $\overline{u} \in H^1_0(\Omega)$ there holds
  \begin{equation}\label{regularization L2 H1}
    \| u_\varrho - \overline{u} \|_{L^2(\Omega)} \leq \varrho^{1/4} \,
    \| \nabla \overline{u} \|_{L^2(\Omega)} .
  \end{equation}
  Moreover, if $\overline{u} \in H^1_0(\Omega)$ satisfies
  $\Delta \overline{u} \in L^2(\Omega)$, then
  \begin{equation}\label{regularization L2 H2}
    \| u_\varrho - \overline{u} \|_{L^2(\Omega)} \leq\varrho^{1/2} \,
    \| \Delta \overline{u} \|_{L^2(\Omega)} .
  \end{equation}
\end{lemma}

\begin{proof}
When considering the variational formulation \eqref{VF adjoint} for
 $q=u_\varrho$, and \eqref{VF primal} for $v=p_\varrho$ this gives
\[
  \langle u_\varrho - \overline{u} , u_\varrho \rangle_{L^2(\Omega)} =
  \langle \nabla p_\varrho , \nabla u_\varrho \rangle_{L^2(\Omega)} = -
  \frac{1}{\varrho} \, \langle p_\varrho , p_\varrho \rangle_{L^2(\Omega)},
\]
i.e., we can write
\[
  \frac{1}{\varrho} \, \| p_\varrho \|^2_{L^2(\Omega)} +
  \| u_\varrho - \overline{u} \|^2_{L^2(\Omega)} =
  \langle \overline{u} - u_\varrho , \overline{u} \rangle_{L^2(\Omega)} \leq
  \| u_\varrho - \overline{u} \|_{L^2(\Omega)} \| \overline{u} \|_{L^2(\Omega)},
\]
and therefore, \eqref{regularization L2 L2} follows.

Next we consider the case $\overline{u} \in H^1_0(\Omega)$. Then we can
use \eqref{VF adjoint} for $q=u_\varrho -\overline{u}$
and \eqref{VF primal} for $v=p_\varrho$ to obtain
\begin{eqnarray*}
  \| u_\varrho - \overline{u} \|^2_{L^2(\Omega)}
  & = & \langle u_\varrho - \overline{u} ,
        u_\varrho - \overline{u} \rangle_{L^2(\Omega)} 
        \, = \, \langle \nabla p_\varrho , \nabla (u_\varrho-\overline{u})
        \rangle_{L^2(\Omega)} \\
  && \hspace*{-2.8cm}
     = \, \langle \nabla p_\varrho , \nabla u_\varrho \rangle_{L^2(\Omega)} -
        \langle \nabla p_\varrho , \nabla \overline{u} \rangle_{L^2(\Omega)} \,
  = \, - \frac{1}{\varrho} \, \langle p_\varrho , p_\varrho \rangle_{L^2(\Omega)} -
        \langle \nabla p_\varrho , \nabla \overline{u} \rangle_{L^2(\Omega)} ,
\end{eqnarray*}
i.e.,
\begin{equation}\label{Proof L2 H1}
  \| u_\varrho - \overline{u} \|_{L^2(\Omega)}^2 +
  \frac{1}{\varrho} \, \| p_\varrho \|^2_{L^2(\Omega)} = -
  \langle \nabla p_\varrho , \nabla \overline{u} \rangle_{L^2(\Omega)} \leq
  \| \nabla p_\varrho \|_{L^2(\Omega)} \| \nabla \overline{u} \|_{L^2(\Omega)} .
\end{equation}
Now, using \eqref{VF adjoint} for $q =p_\varrho$ this gives
\[
  \| \nabla p_\varrho \|_{L^2(\Omega)}^2 =
  \langle \nabla p_\varrho , \nabla p_\varrho \rangle_{L^2(\Omega)} =
  \langle u_\varrho - \overline{u} , p_\varrho \rangle_{L^2(\Omega)} \leq
  \| u_\varrho - \overline{u} \|_{L^2(\Omega)} \| p_\varrho \|_{L^2(\Omega)},
\]
and hence,
\[
  \| u_\varrho - \overline{u} \|_{L^2(\Omega)}^2 \leq
  \| u_\varrho - \overline{u} \|_{L^2(\Omega)}^{1/2}
  \| p_\varrho \|_{L^2(\Omega)}^{1/2} \| \nabla \overline{u} \|_{L^2(\Omega)}
\]
follows, i.e.,
\begin{equation}\label{Proof L2 H1 2}
  \| u_\varrho - \overline{u} \|_{L^2(\Omega)} \leq
  \| p_\varrho \|_{L^2(\Omega)}^{1/3}
  \| \nabla \overline{u} \|^{2/3}_{L^2(\Omega)} .
\end{equation}
Moreover, \eqref{Proof L2 H1} then implies
\begin{eqnarray*}
  \| p_\varrho \|^2_{L^2(\Omega)}
  & \leq & \varrho \, \| \nabla p_\varrho \|_{L^2(\Omega)}
           \| \nabla \overline{u} \|_{L^2(\Omega)} \\
  & \leq & \varrho \, \| u_\varrho - \overline{u} \|_{L^2(\Omega)}^{1/2}
           \| p_\varrho \|_{L^2(\Omega)}^{1/2}
           \| \nabla \overline{u} \|_{L^2(\Omega)} \\
  & \leq & \varrho \, \| p_\varrho \|_{L^2(\Omega)}^{2/3}
           \| \nabla \overline{u} \|_{L^2(\Omega)}^{4/3},
\end{eqnarray*}
i.e.,
\[
  \| p_\varrho \|_{L^2(\Omega)} \leq \varrho^{3/4} \,
  \| \nabla \overline{u} \|_{L^2(\Omega)} .
\]
Now, \eqref{regularization L2 H1} follows when using
\eqref{Proof L2 H1 2}.

Finally we assume $\overline{u} \in H^1_0(\Omega)$ satisfying
$\Delta \overline{u} \in L^2(\Omega)$. As in the derivation of
\eqref{Proof L2 H1} we have
\[
 \| u_\varrho - \overline{u} \|^2_{L^2(\Omega)}
 = \langle \nabla p_\varrho , \nabla u_\varrho \rangle_{L^2(\Omega)} -
 \langle \nabla p_\varrho , \nabla \overline{u} \rangle_{L^2(\Omega)} .
\]
Now, using integration by parts, inserting
$ p_\varrho = - \varrho z_\varrho$, and using $z_\varrho = -\Delta u_\varrho$,
this gives
\begin{eqnarray*}
  \| u_\varrho - \overline{u} \|^2_{L^2(\Omega)}
  & = & - \langle p_\varrho , \Delta u_\varrho \rangle_{L^2(\Omega)} +
        \langle p_\varrho , \Delta \overline{u} \rangle_{L^2(\Omega)} \\
  & = & \varrho \, \langle z_\varrho , \Delta u_\varrho \rangle_{L^2(\Omega)} -
        \varrho \, \langle z_\varrho , \Delta \overline{u} \rangle_{L^2(\Omega)} \\
  & = & - \varrho \, \langle \Delta u_\varrho ,
        \Delta u_\varrho \rangle_{L^2(\Omega)} +
        \varrho \, \langle \Delta u_\varrho ,
        \Delta \overline{u} \rangle_{L^2(\Omega)},
\end{eqnarray*}
i.e.,
\[
  \| u_\varrho - \overline{u} \|_{L^2(\Omega)}^2 + \varrho \,
  \| \Delta u_\varrho \|_{L^2(\Omega)}^2 =
  \varrho \, \langle \Delta u_\varrho , \Delta
  \overline{u} \rangle_{L^2(\Omega)} \leq
  \varrho \, \| \Delta u_\varrho \|_{L^2(\Omega)}
  \| \Delta \overline{u} \|_{L^2(\Omega)},
\]
and hence,
\[
  \| \Delta u_\varrho \|_{L^2(\Omega)} \leq \| \Delta \overline{u} \|_{L^2(\Omega)},
  \quad
  \| u_\varrho - \overline{u} \|^2_{L^2(\Omega)} \leq \varrho \,
  \| \Delta \overline{u} \|_{L^2(\Omega)}^2 ,
\]
i.e., \eqref{regularization L2 H2} follows.
\end{proof}

Note that the regularization error estimates of
Lemma \ref{Lemma Regularization L2} were already given in
\cite[Theorem 4.1]{LLSY:NeumuellerSteinbach:2021a}. But in particular the proof
of \eqref{regularization L2 H1} is a bit different to that of
\cite{LLSY:NeumuellerSteinbach:2021a}, resulting in an improved constant,
and \eqref{regularization L2 H2} is new. In addition to regularization
error estimates in $L^2(\Omega)$ we also need to have related estimates
in $H^1_0(\Omega)$ when assuming $\overline{u} \in H^1_0(\Omega)$.

\begin{lemma}
  Let $(u_\varrho,p_\varrho) \in H^1_0(\Omega) \times H^1_0(\Omega)$ be the
  unique solution of the coupled variational formulation
  \eqref{VF primal} and \eqref{VF adjoint}. When assuming
  $\overline{u} \in H^1_0(\Omega)$, this gives
  \begin{equation}\label{regularization H1 H1}
    \| \nabla (u_\varrho - \overline{u}) \|_{L^2(\Omega)} \leq
    \| \nabla \overline{u} \|_{L^2(\Omega)} .
  \end{equation}
  Moreover, if $\overline{u} \in H^1_0(\Omega)$ satisfies
  $\Delta \overline{u} \in L^2(\Omega)$, then
  \begin{equation}\label{regularization H1 H2}
    \| \nabla (u_\varrho - \overline{u} ) \|_{L^2(\Omega)} \leq
    \varrho^{1/4} \, \| \Delta \overline{u} \|_{L^2(\Omega)} .
  \end{equation}
\end{lemma}

\begin{proof}
  Due to the assumption $\overline{u} \in H^1_0(\Omega)$ we can use
  $v = u_\varrho - \overline{u}$ as test function in \eqref{VF primal}
  and $q=p_\varrho$ in \eqref{VF adjoint} to obtain
  \[
    \langle \nabla u_\varrho ,
    \nabla (u_\varrho-\overline{u}) \rangle_{L^2(\Omega)} =
    - \frac{1}{\varrho} \, \langle p_\varrho ,
    u_\varrho - \overline{u} \rangle_{L^2(\Omega)} =
    - \frac{1}{\varrho} \, \langle \nabla p_\varrho ,
    \nabla_\varrho p \rangle_{L^2(\Omega)},
  \]
  i.e., we have
  \begin{eqnarray*}
    \frac{1}{\varrho} \, \| \nabla p_\varrho \|^2_{L^2(\Omega)} +
    \| \nabla (u_\varrho - \overline{u}) \|^2_{L^2(\Omega)}
    & = & \langle \nabla \overline{u} ,
          \nabla (\overline{u} - u_\varrho) \rangle_{L^2(\Omega)} \\
    & \leq & \| \nabla \overline{u} \|_{L^2(\Omega)}
             \| \nabla (u_\varrho - \overline{u}) \|_{L^2(\Omega)} ,
  \end{eqnarray*}
  from which we conclude \eqref{regularization H1 H1}.

  On the other hand, when assuming $\overline{u} \in H^1_0(\Omega)$
  satisfying $\Delta \overline{u} \in L^2(\Omega)$,
  and when applying integration by parts, we also have
  \begin{eqnarray*}
    \frac{1}{\varrho} \, \| \nabla p_\varrho \|^2_{L^2(\Omega)} +
    \| \nabla (u_\varrho - \overline{u}) \|^2_{L^2(\Omega)}
    & = & \langle \nabla \overline{u} ,
          \nabla (\overline{u} - u_\varrho) \rangle_{L^2(\Omega)} \\
    & & \hspace*{-3cm} = \, \langle \Delta \overline{u} ,
          u_\varrho - \overline{u} \rangle_{L^2(\Omega)}
          \leq \| \Delta \overline{u} \|_{L^2(\Omega)}
          \| u_\varrho - \overline{u} \|_{L^2(\Omega)} .
\end{eqnarray*}
Finally, using \eqref{regularization L2 H2}, this gives
\[
  \| \nabla (u_\varrho - \overline{u}) \|^2_{L^2(\Omega)}
  \leq \| \Delta \overline{u} \|_{L^2(\Omega)}
  \| u_\varrho - \overline{u} \|_{L^2(\Omega)} \leq \varrho^{1/2} \,
  \| \Delta \overline{u} \|^2_{L^2(\Omega)},
\]
i.e., \eqref{regularization H1 H2} follows.
\end{proof}

%
%

\section{Finite element error estimates}
\label{sec:FiniteElementErrorEstimates}

For the numerical solution of the coupled variational formulation
\eqref{VF primal} and \eqref{VF adjoint} we first introduce the
transformation $p_\varrho(x) = \sqrt{\varrho} \, \widetilde{p}_\varrho(x)$,
i.e., we consider the variational formulation to find
$(u_\varrho,\widetilde{p}_\varrho) \in H^1_0(\Omega) \times H^1_0(\Omega)$
such that
\begin{equation}\label{VF primal scaling}
  \frac{1}{\sqrt{\varrho}} \, \langle \widetilde{p}_\varrho ,
  v \rangle_{L^2(\Omega)} +
  \langle \nabla u_\varrho , \nabla v \rangle_{L^2(\Omega)} = 0 \quad
  \mbox{for all} \; v \in H^1_0(\Omega),
\end{equation}
and
\begin{equation}\label{VF adjoint scaling}
  - \langle \nabla \widetilde{p}_\varrho , \nabla q \rangle_{L^2(Q)}
  + \frac{1}{\sqrt{\varrho}} \, \langle u_\varrho , q \rangle_{L^2(\Omega)}
  = \frac{1}{\sqrt{\varrho}}
  \langle \overline{u} , q \rangle_{L^2(Q)} \quad \mbox{for all} \;
  q \in H^1_0(\Omega) .
\end{equation}
Let $V_h = S_h^1(\Omega) \cap H^1_0(\Omega)
= \mbox{span} \{ \varphi_{h,k}\}_{k=1}^{N_h} 
= \mbox{span} \{ \varphi_{k}\}_{k=1}^{N_h}$ 
be the finite element space of piecewise linear and continuous basis
functions which are defined with respect
to some admissible decomposition of the computational domain $\Omega$
into shape regular and globally quasi-uniform simplicial finite elements
of mesh size $h$. For simplicity, we omit the subindex $h$ from the
basis functions $\varphi_{h,k}$.
The Galerkin variational formulation of
\eqref{VF primal scaling} and \eqref{VF adjoint scaling} is to find
$(u_{\varrho h},\widetilde{p}_{\varrho h}) \in V_h \times V_h$ such that
\begin{equation}\label{FEM primal}
  \frac{1}{\sqrt{\varrho}} \,
  \langle \widetilde{p}_{\varrho h} , v_h \rangle_{L^2(\Omega)} +
  \langle \nabla u_{\varrho h} , \nabla v_h \rangle_{L^2(\Omega)} = 0 \quad
  \mbox{for all} \; v_h \in V_h,
\end{equation}
and
\begin{equation}\label{FEM adjoint}
  - \langle \nabla \widetilde{p}_{\varrho h} , \nabla q_h \rangle_{L^2(Q)} +
  \frac{1}{\sqrt{\varrho}} \, \langle u_{\varrho h} , q_h \rangle_{L^2(\Omega)} =
  \frac{1}{\sqrt{\varrho}} \,
  \langle \overline{u} , q_h \rangle_{L^2(Q)} \quad \mbox{for all} \;
  q_h \in V_h .
\end{equation}
The mixed finite element scheme \eqref{FEM primal} and \eqref{FEM adjoint}
has obviously a unique solution. Indeed, choosing the test function
$v_h = \widetilde{p}_{\varrho h}$ in \eqref{FEM primal} as well as
$q_h = u_{\varrho h}$ in \eqref{FEM adjoint}, and adding both equations,
we see that $\widetilde{p}_{\varrho h}$ and $u_{\varrho h}$ must be zero
for the homogeneous equations ($\overline{u}=0$).
Now uniqueness always yields existence in the linear finite-dimensional case.
So, the corresponding system of algebraic equation also has a unique solution 
and vice versa; see also Section~\ref{sec:RobustSolvers}.

Now we are in a position to formulate the main result of this section.

\begin{theorem}
  Let $(u_{\varrho h}, \widetilde{p}_{\varrho h}) \in V_h \times V_h$ be
  the unique solution of the coupled finite element variational formulation
  \eqref{FEM primal} and \eqref{FEM adjoint}. Assume that the underlying
  finite element mesh is globally quasi-uniform such that an inverse
  inequality in $V_h$ is valid, and consider $\varrho = h^4$.
  For $\overline{u} \in H^1_0(\Omega)$ then there holds the error estimate
  \begin{equation}\label{Final error L2 H1}
    \| u_{\varrho h} - \overline{u} \|_{L^2(\Omega)} \, \leq \, c \, h \,
    |\overline{u}|_{H^1(\Omega)} .
  \end{equation}
  For $\overline{u} \in H^1_0(\Omega) \cap H^2(\Omega)$, and if the
  domain $\Omega$ is either smoothly bounded or convex,
  then we also have
  \begin{equation}\label{Final error L2 H2}
    \| u_{\varrho h} - \overline{u} \|_{L^2(\Omega)} \leq c \, h^2 \,
    | \overline{u} |_{H^2(\Omega)} .
  \end{equation}  
\end{theorem}

\begin{proof}
  For given $(\varphi,\psi) \in H^1_0(\Omega) \times H^1_0(\Omega)$ define
  $(\varphi_h,\psi_h) \in V_h \times V_h$ as unique solutions satisfying
  the variational formulations
  \[
    \frac{1}{\sqrt{\varrho}} \, \langle \psi_h , v_h \rangle_{L^2(\Omega)} +
    \langle \nabla \varphi_h , \nabla v_h \rangle_{L^2(\Omega)} =
    \frac{1}{\sqrt{\varrho}} \, \langle \psi , v_h \rangle_{L^2(\Omega)} +
    \langle \nabla \varphi , \nabla v_h \rangle_{L^2(\Omega)}, \;
    \forall v_h \in V_h,
  \]
  and
  \[
    - \langle \nabla \psi_h , \nabla q_h \rangle_{L^2(\Omega)} +
    \frac{1}{\sqrt{\varrho}} \, \langle \varphi_h , q_h \rangle_{L^2(\Omega)}
    =
    - \langle \nabla \psi , \nabla q_h \rangle_{L^2(\Omega)} +
    \frac{1}{\sqrt{\varrho}} \, \langle \varphi , q_h \rangle_{L^2(\Omega)},
    q_h \in V_h.
  \]
  When using an inverse inequality in $V_h$, this gives
  \begin{eqnarray*}
    && \hspace*{-5mm}
       \frac{1}{\sqrt{\varrho}} \, \| \varphi_h \|^2_{L^2(\Omega)} +
       \| \nabla \varphi_h \|^2_{L^2(\Omega)} +
       \frac{1}{\sqrt{\varrho}} \, \| \psi_h \|^2_{L^2(\Omega)} +
       \| \nabla \psi_h \|^2_{L^2(\Omega)} \\
    && \hspace*{-3mm} \leq \,
       \left( \frac{1}{\sqrt{\varrho}} + c_I \, h^{-2} \right)
       \| \varphi_h \|^2_{L^2(\Omega)} +
       \left( \frac{1}{\sqrt{\varrho}} + c_I \, h^{-2} \right)
       \| \psi_h \|^2_{L^2(\Omega)} \\
    && \hspace*{-3mm}
       = \, \Big( 1+ c_I \, h^{-2} \, \sqrt{\varrho} \Big)
       \left[ \frac{1}{\sqrt{\varrho}}
       \| \varphi_h \|^2_{L^2(\Omega)} + \frac{1}{\sqrt{\varrho}} 
       \| \psi_h \|^2_{L^2(\Omega)} \right] \\
    && \hspace*{-3mm}
       = \, \Big( 1 + c_I \, h^{-2} \, \sqrt{\varrho} \Big)
       \left[ \frac{1}{\sqrt{\varrho}}
       \langle \varphi_h , \varphi_h \rangle_{L^2(\Omega)}
       - \langle \nabla \psi_h , \nabla \varphi_h \rangle_{L^2(\Omega)}
       \right. \\
    && \hspace*{3cm} \left.
       + \langle \nabla \varphi_h , \nabla \psi_h \rangle_{L^2(\Omega)}
       + \frac{1}{\sqrt{\varrho}} 
       \langle \psi_h , \psi_h \rangle_{L^2(\Omega)} \right] \\
    && \hspace*{-3mm}
       = \, \Big( 1 + c_I \, h^{-2} \, \sqrt{\varrho} \Big)
       \left[ \frac{1}{\sqrt{\varrho}}
       \langle \varphi , \varphi_h \rangle_{L^2(\Omega)}
       - \langle \nabla \psi , \nabla \varphi_h \rangle_{L^2(\Omega)}
       \right. \\
    && \hspace*{3cm} \left.
       + \langle \nabla \varphi , \nabla \psi_h \rangle_{L^2(\Omega)}
       + \frac{1}{\sqrt{\varrho}} 
       \langle \psi , \psi_h \rangle_{L^2(\Omega)} \right] \\
    && \hspace*{-3mm}
       \leq \, \Big( 1 + c_I \, h^{-2} \, \sqrt{\varrho} \Big)
       \left[ \frac{1}{\sqrt{\varrho}} \, \| \varphi \|_{L^2(\Omega)}
       \| \varphi_h \|_{L^2(\Omega)} + \| \nabla \psi \|_{L^2(\Omega)}
       \| \nabla \varphi_h \|^2_{L^2(\Omega)} \right. \\
  && \hspace*{3cm} \left.
     + \| \nabla \varphi \|_{L^2(\Omega)} \| \nabla \psi_h \|_{L^2(\Omega)}
     + \frac{1}{\sqrt{\varrho}} \, \| \psi \|_{L^2(\Omega)}
     \| \psi_h \|_{L^2(\Omega)}
     \right] \\
    && \hspace*{-3mm}
       \leq \, \Big( 1+ c_I \, h^{-2} \, \sqrt{\varrho} \Big)
       \left[ \frac{1}{\sqrt{\varrho}} \, \| \varphi \|_{L^2(\Omega)}^2 +
       \| \nabla \varphi \|^2_{L^2(\Omega)} + \frac{1}{\sqrt{\varrho}} \,
       \| \psi \|_{L^2(\Omega)}^2 +
       \| \nabla \psi \|_{L^2(\Omega)}^2 \right]^{1/2} \\
  && \hspace*{1.5cm} \cdot
     \left[
     \frac{1}{\sqrt{\varrho}} \, \| \varphi_h \|_{L^2(\Omega)}^2 +
     \| \nabla \varphi_h \|^2_{L^2(\Omega)} + \frac{1}{\sqrt{\varrho}} \,
     \| \psi_h \|_{L^2(\Omega)}^2 +
     \| \nabla \psi_h \|_{L^2(\Omega)}^2 \right]^{1/2}.
\end{eqnarray*}
Hence,
\begin{eqnarray*}
&& \hspace*{-5mm} \frac{1}{\sqrt{\varrho}} \, \| \varphi_h \|_{L^2(\Omega)}^2 +
     \| \nabla \varphi_h \|^2_{L^2(\Omega)} + \frac{1}{\sqrt{\varrho}} \,
     \| \psi_h \|_{L^2(\Omega)}^2 +
   \| \nabla \psi_h \|_{L^2(\Omega)}^2 \\
  && \hspace*{-3mm} \leq \,
     \Big( 1+ c_I \, h^{-2} \, \sqrt{\varrho} \Big)^2 \left[
     \frac{1}{\sqrt{\varrho}} \, \| \varphi \|_{L^2(\Omega)}^2 +
     \| \nabla \varphi \|^2_{L^2(\Omega)} + \frac{1}{\sqrt{\varrho}} \,
     \| \psi \|_{L^2(\Omega)}^2 +
     \| \nabla \psi \|_{L^2(\Omega)}^2 \right] .
\end{eqnarray*}
In fact, the Galerkin projection $(\varphi,\psi) \mapsto (\varphi_h,\psi_h)$
is bounded, and in particular for $\varrho = h^4$ we therefore conclude
Cea's lemma, i.e., for arbitrary $(v_h,q_h) \in V_h \times V_h$ we have
\begin{eqnarray*}
&& h^{-2} \, \| u_\varrho -u_{\varrho h} \|_{L^2(\Omega)}^2 +
   \| \nabla (u_\varrho -u_{\varrho h}) \|^2_{L^2(\Omega)} \\
  &&  \hspace*{5mm} + \, h^{-2} \,
     \| \widetilde{p}_\varrho - \widetilde{p}_{\varrho h} \|_{L^2(\Omega)}^2 +
   \| \nabla (\widetilde{p}_\varrho - \widetilde{p}_{\varrho h})
   \|_{L^2(\Omega)}^2 \\
  && \hspace*{10mm} \leq \,
     c \cdot  \left[
     h^{-2} \, \| u_\varrho - v_h \|_{L^2(\Omega)}^2 +
     \| \nabla (u_\varrho - v_h)\|^2_{L^2(\Omega)} \right. \\
  && \hspace*{20mm} \left. + \, h^{-2} \,
     \| \widetilde{p}_\varrho - q_h \|_{L^2(\Omega)}^2 +
     \| \nabla ( \widetilde{p}_\varrho - q_h) \|_{L^2(\Omega)}^2 \right] .
\end{eqnarray*}
For $\overline{u} \in H^1_0(\Omega)$ we can consider $v_h=\overline{u}_h$
being a suitable approximation, e.g.,
Scott--Zhang interpolation \cite{BrennerScott}.
Then, using \eqref{regularization L2 H1},
this gives, recall $\varrho = h^4$,
\begin{eqnarray*}
  \| u_\varrho - \overline{u}_h \|_{L^2(\Omega)}
  & \leq & \| u_\varrho - \overline{u} \|_{L^2(\Omega)} +
           \| \overline{u} - \overline{u}_h \|_{L^2(\Omega)} \\
  & \leq & \varrho^{1/4} \, | \overline{u} |_{H^1(\Omega)} +
           c \, h \, |\overline{u}|_{H^1(\Omega)} \leq
           c \, h \, |\overline{u}|_{H^1(\Omega)} .
\end{eqnarray*}
Moreover, now using \eqref{regularization H1 H1}, we also have
\[
  \| \nabla(u_\varrho - \overline{u}_h) \|_{L^2(\Omega)}
  \leq \| \nabla(u_\varrho - \overline{u}) \|_{L^2(\Omega)} +
       \| \nabla (\overline{u} - \overline{u}_h) \|_{L^2(\Omega)} 
  \leq (1+c) \, \| \nabla \overline{u} \|_{L^2(\Omega)} \, .
\]
Correspondingly, let $q_h = \Pi_h \widetilde{p}_\varrho$ be some
suitable approximation, e.g., again the Scott--Zhang interpolation.
Then,
\[
  \| \widetilde{p}_\varrho - \Pi_h \widetilde{p}_\varrho \|_{L^2(\Omega)}
  \leq c \, h \, \| \nabla \widetilde{p}_\varrho \|_{L^2(\Omega)}
           \, = \, c \, h \, \frac{1}{\sqrt{\varrho}} \,
           \| \nabla p_\varrho \|_{L^2(\Omega)} \, .
\]
From \eqref{VF adjoint}, and using duality we first have
\[
  \| \nabla p_\varrho \|^2_{L^2(\Omega)}
  = \langle \nabla p_\varrho , \nabla p_\varrho \rangle_{L^2(\Omega)}
  = 
  \langle u_\varrho - \overline{u} , p_\varrho \rangle_{L^2(\Omega)} 
  \leq \| u_\varrho - \overline{u} \|_{H^{-1}(\Omega)(\Omega)}
  \| \nabla p_\varrho \|_{L^2(\Omega)},
\]
and with \cite[Theorem 4.1]{LLSY:NeumuellerSteinbach:2021a} this gives
\[
  \| \nabla p_\varrho \|_{L^2(\Omega)} \leq 
  \| u_\varrho - \overline{u} \|_{H^{-1}(\Omega)} \leq
  \sqrt{\varrho} \, | \overline{u} |_{H^1(\Omega)} .
\]
Hence we conclude
\[
  \| \widetilde{p}_\varrho - \Pi_h \widetilde{p}_\varrho \|_{L^2(\Omega)}
  \leq c \, h \, |\overline{u}|_{H^1(\Omega)} .
\]
In the same way as above we also obtain
\[
  \| \nabla (\widetilde{p}_\varrho - \Pi_h \widetilde{p}_\varrho)
  \|_{L^2(\Omega)} \leq \| \nabla \widetilde{p}_\varrho \|_{L^2(\Omega)}
  = \frac{1}{\sqrt{\varrho}} \, \| \nabla p_\varrho \|_{L^2(\Omega)}
  \leq |\overline{u}|_{H^1(\Omega)} .
\]
Now, summing up all contributions this gives
\begin{eqnarray*}
&& h^{-2} \, \| u_\varrho -u_{\varrho h} \|_{L^2(\Omega)}^2 +
   \| \nabla (u_\varrho -u_{\varrho h}) \|^2_{L^2(\Omega)} \\
  &&  \hspace*{5mm} + \, h^{-2} \,
     \| \widetilde{p}_\varrho - \widetilde{p}_{\varrho h} \|_{L^2(\Omega)}^2 +
   \| \nabla (\widetilde{p}_\varrho - \widetilde{p}_{\varrho h})
   \|_{L^2(\Omega)}^2 \, \leq \, c \, |\overline{u}|_{H^1(\Omega)}^2 ,
\end{eqnarray*}
in particular we have
\[
  \| u_\varrho -u_{\varrho h} \|_{L^2(\Omega)}^2 \leq
  c \, h^2 \, |\overline{u}|_{H^1(\Omega)}^2 .
\]
Hence, \eqref{Final error L2 H1} follows from, again using
\eqref{regularization L2 H1} and $\varrho = h^4$,
\[
  \| u_{\varrho h} - \overline{u} \|_{L^2(\Omega)} \leq
  \| u_{\varrho h} - u_\varrho \|_{L^2(\Omega)} +
  \| u_\varrho - \overline{u} \|_{L^2(\Omega)} \leq
  c \, h \, |\overline{u}|_{H^1(\Omega)} + \varrho^{1/4} \,
  |\overline{u}|_{H^1(\Omega)} .
\]
It remains to consider $\overline{u} \in H^1_0(\Omega) \cap H^2(\Omega)$.
We proceed as above, but with \eqref{regularization L2 H2} we first obtain
\begin{eqnarray*}
  \| u_\varrho - \overline{u}_h \|_{L^2(\Omega)}
  & \leq & \| u_\varrho - \overline{u} \|_{L^2(\Omega)} +
           \| \overline{u} - \overline{u}_h \|_{L^2(\Omega)} \\
  & \leq & \varrho^{1/2} \, \| \Delta \overline{u} \|_{L^2(\Omega)} +
           c \, h^2 \, |\overline{u}|_{H^2(\Omega)} \leq
           c \, h^2 \, |\overline{u}|_{H^2(\Omega)} ,
\end{eqnarray*}
while with \eqref{regularization H1 H2} we conclude
\begin{eqnarray*}
  \| \nabla(u_\varrho - \overline{u}_h) \|_{L^2(\Omega)}
  & \leq & \| \nabla(u_\varrho - \overline{u}) \|_{L^2(\Omega)} +
           \| \nabla (\overline{u} - \overline{u}_h) \|_{L^2(\Omega)} \\ 
  & \leq & \varrho^{1/4} \, \| \Delta \overline{u} \|_{L^2(\Omega)}
           + c \, h \, | \overline{u} |_{H^2(\Omega)} \, \leq \,
           c \, h \, |\overline{u}|_{H^2(\Omega)} .
\end{eqnarray*}
Moreover, we also have, in the case of a smoothly bounded or convex domain
$\Omega$, and using \eqref{regularization L2 H2},
\begin{eqnarray*}
  \| \widetilde{p}_\varrho - \Pi_h \widetilde{p}_\varrho \|_{L^2(\Omega)}
  & \leq & c \, h^2 \, | \widetilde{p}_\varrho |_{H^2(\Omega)} =
           c \, h^2 \, \frac{1}{\sqrt{\varrho}} \, |p_\varrho|_{H^2(\Omega)}
           \leq c \, h^2 \, \frac{1}{\sqrt{\varrho}} \,
           \|\Delta p_\varrho \|_{L^2(\Omega)} \\
  & = & c \, h^2 \, \frac{1}{\sqrt{\varrho}} \,
      \| u_\varrho - \overline{u} \|_{L^2(\Omega)}
      \leq c \, h^2 \, \| \Delta \overline{u} \|_{L^2(\Omega)}.
\end{eqnarray*}
Finally, and following the above estimates, we have
\[
  \| \nabla (\widetilde{p}_\varrho - \Pi_h \widetilde{p}_\varrho )
  \|_{L^2(\Omega)}
  \, \leq \, c \, h \, |\widetilde{p}_\varrho |_{H^2(\Omega)} \, = \, 
           c \, h \, \frac{1}{\sqrt{\varrho}} \, |p_\varrho|_{H^2(\Omega)} 
           \leq \, c \, h \, \|\Delta \overline{u}\|_{L^2(\Omega)} .
\]
Again, summing up all contributions this gives
\begin{eqnarray*}
&& h^{-2} \, \| u_\varrho -u_{\varrho h} \|_{L^2(\Omega)}^2 +
   \| \nabla (u_\varrho -u_{\varrho h}) \|^2_{L^2(\Omega)} \\
  &&  \hspace*{5mm} + \, h^{-2} \,
     \| \widetilde{p}_\varrho - \widetilde{p}_{\varrho h} \|_{L^2(\Omega)}^2 +
   \| \nabla (\widetilde{p}_\varrho - \widetilde{p}_{\varrho h})
   \|_{L^2(\Omega)}^2 \, \leq \, c \, h^2 |\overline{u}|_{H^2(\Omega)}^2 ,
\end{eqnarray*}
i.e.,
\[
  \| u_\varrho - u_{\varrho h} \|^2_{L^2(\Omega)} \leq c \, h^4 \,
  |\overline{u}|_{H^2(\Omega)} .
\]
Now, \eqref{Final error L2 H2} follows from the triangle
inequality and using \eqref{regularization H1 H2}.
\end{proof}

In order to derive error estimates for less regular targets
$\overline{u}$ we also need the following result.

\begin{lemma}
  Let $(u_{\varrho h}, \widetilde{p}_{\varrho h}) \in V_h \times V_h$ be
  the unique solution of the coupled finite element variational formulation
  \eqref{FEM primal} and \eqref{FEM adjoint}.
  For $\overline{u} \in L^2(\Omega)$ there holds the error estimate
  \begin{equation}\label{Final error L2 L2}
    \| u_{\varrho h} - \overline{u} \|_{L^2(\Omega)} \leq
    \| \overline{u} \|_{L^2(\Omega)} .
  \end{equation}
\end{lemma}
\begin{proof}
  We consider the Galerkin formulations \eqref{FEM primal} and
  \eqref{FEM adjoint} for the particular test functions
  $v_h = \widetilde{p}_{\varrho h}$ and $q_h = u_{\varrho h}$ to obtain
  \[
    \frac{1}{\sqrt{\varrho}} \,
    \langle \widetilde{p}_{\varrho h} ,
    \widetilde{p}_{\varrho h} \rangle_{L^2(\Omega)} +
    \langle \nabla u_{\varrho h} ,
    \nabla \widetilde{p}_{\varrho h} \rangle_{L^2(\Omega)} = 0,
  \]
  and
  \[
    - \langle \nabla \widetilde{p}_{\varrho h} ,
    \nabla u_{\varrho h} \rangle_{L^2(Q)}
    +
    \frac{1}{\sqrt{\varrho}} \, \langle u_{\varrho h} ,
    u_{\varrho h} \rangle_{L^2(\Omega)} = \frac{1}{\sqrt{\varrho}}
    \langle \overline{u} , u_{\varrho h} \rangle_{L^2(Q)} .
  \]
  When summing up both expressions this gives
  \[  
    \langle \widetilde{p}_{\varrho h} ,
    \widetilde{p}_{\varrho h} \rangle_{L^2(\Omega)} +
    \langle u_{\varrho h} , u_{\varrho h} \rangle_{L^2(\Omega)} =
    \langle \overline{u} , u_{\varrho h} \rangle_{L^2(\Omega)},
  \]
  which we can write as
  \[  
    \langle \widetilde{p}_{\varrho h} ,
    \widetilde{p}_{\varrho h} \rangle_{L^2(\Omega)} +
    \langle u_{\varrho h} - \overline{u},
    u_{\varrho h} - \overline{u} \rangle_{L^2(\Omega)} =
    \langle \overline{u} - u_{\varrho h},
    \overline{u} \rangle_{L^2(\Omega)} .
  \]
  From this we conclude \eqref{Final error L2 L2}.
\end{proof}

Now, using an interpolation argument, we can prove the final error estimate.

\begin{cor}
  Let $(u_{\varrho h}, \widetilde{p}_{\varrho h}) \in V_h \times V_h$ be
  the unique solution of the coupled finite element variational formulation
  \eqref{FEM primal} and \eqref{FEM adjoint}. Let
  $\overline{u} \in H^s_0(\Omega)$ for $ s \in [0,1]$ or
  $\overline{u} \in H^1_0(\Omega) \cap H^s(\Omega)$ for $s \in (1,2]$,
  and let $\varrho = h^4$. Then,
  \begin{equation}\label{Final error L2 Hs}
    \| u_{\varrho h} - \overline{u} \|_{L^2(\Omega)} \leq c \, h^s \,
    \| \overline{u} \|_{H^s(\Omega)} .
  \end{equation}
\end{cor}

%
%

\section{Robust solvers}
\label{sec:RobustSolvers}
The finite element variational formulation \eqref{FEM primal} and
\eqref{FEM adjoint} is equivalent to a coupled linear system of
algebraic equations,
\begin{equation}\label{coupled LGS}
  \frac{1}{\sqrt{\varrho}} \, \mathbf{M}_h \mathbf{\widetilde{p}} +
  \mathbf{K}_h \mathbf{u} = \mathbf{0}, \quad \text{and} \quad
  - \mathbf{K}_h \mathbf{\widetilde{p}} +
  \frac{1}{\sqrt{\varrho}} \, \mathbf{M}_h
  \mathbf{u} = \frac{1}{\sqrt{\varrho}} \, \mathbf{f} ,
\end{equation}
where $\mathbf{K}_h$ and $\mathbf{M}_h$ are the standard finite element
stiffness and mass matrices, the entries of which are given by 
\[
  \mathbf{K}_h[\ell,k] =
  \int_\Omega \nabla \varphi_k \cdot \nabla \varphi_\ell \, dx
  \quad \text{and} \quad 
   \mathbf{M}_h[\ell,k] = \int_\Omega \varphi_k \, \varphi_\ell \, dx \quad 
  \mbox{for} \; k,\ell=1,\ldots,N_h,
\]
respectively. The load vector $\mathbf{f} = (f_\ell)_{\ell=1,\ldots,N_h}
  \in \mathbb{R}^{N_h} $ is given by its entries
\[
  f_\ell = \int_\Omega \overline{u} \, \varphi_\ell \, dx \quad
  \mbox{for} \; \ell=1,\ldots,N_h.
\]
The stiffness matrix $\mathbf{K}_h$ and the mass matrix $\mathbf{M}_h$ 
are symmetric and positive definite. Furthermore, they satisfy the spectral 
inequalities
\begin{equation}
\label{eqn:SpectralEquivalenceInequalitiesK}
\underline{c}_\text{K} h^d (\mathbf{v},\mathbf{v}) 
\le \lambda_{\text{min}}(\mathbf{K})  (\mathbf{v},\mathbf{v}) 
\le (\mathbf{K}\mathbf{v},\mathbf{v})  
\le \lambda_{\text{max}}(\mathbf{K})  (\mathbf{v},\mathbf{v}) 
\le\overline{c}_\text{K} h^{d-2} (\mathbf{v},\mathbf{v}) 
\end{equation}
and
\begin{equation}
\label{eqn:SpectralEquivalenceInequalitiesM}
\underline{c}_\text{M} h^d (\mathbf{v},\mathbf{v}) 
\le \lambda_{\text{min}}(\mathbf{M})  (\mathbf{v},\mathbf{v}) 
\le (\mathbf{M}\mathbf{v},\mathbf{v})  
\le \lambda_{\text{max}}(\mathbf{M})  (\mathbf{v},\mathbf{v}) 
\le \overline{c}_\text{M} h^{d} (\mathbf{v},\mathbf{v}) 
\end{equation}
for all $\mathbf{v} \in \mathbb{R}^{N_h}$, where  $\lambda_{\text{min}}(\cdot)$
and $\lambda_{\text{max}}(\cdot)$ always denote the minimal and maximal 
eigenvalues of the corresponding matrices, respectively, and
$(\cdot,\cdot) = (\cdot,\cdot)_{\mathbb{R}^{N_h}}$ denotes the Euclidean
inner product in the Euclidean vector space $\mathbb{R}^{N_h}$. The
positive \linebreak
constants $\underline{c}_\text{K}$, $\overline{c}_\text{K}$,
$\underline{c}_\text{M}$, and $\overline{c}_\text{M}$ are independent of
the mesh-size $h$; see, e.g., 
\cite{LLSY:ErnGuermond:2004a, LLSY:Steinbach:2008a}.

Since the mass matrix $\mathbf{M}_h$ is invertible, we can eliminate
the modified adjoint $\mathbf{\widetilde{p}}$, 
and hence we can rewrite \eqref{coupled LGS} as Schur complement system
\begin{equation}
\label{eqn:Su=f}
\Big[ \varrho \,\mathbf{K}_h \mathbf{M}_h^{-1} \mathbf{K}_h +
\mathbf{M}_h \Big] \mathbf{u} = \mathbf{f} \, .
\end{equation}
The Schur complement 
$\mathbf{S}_h = \varrho \,\mathbf{K}_h \mathbf{M}_h^{-1} \mathbf{K}_h
+ \mathbf{M}_h$ is obviously symmetric and positive definite, 
and hence invertible. Therefore, the coupled system \eqref{coupled LGS}
also has a unique solution 
$\mathbf{\widetilde{p}} = (\widetilde{p}_k)_{k=1,\ldots,N_h}$,
$\mathbf{{u}} = (u_k)_{k=1,\ldots,N_h} \in \mathbb{R}^{N_h}$
which delivers the nodal parameters for the unique solution
$\widetilde{p}_{\varrho h} = \sum_{k=1}^{N_h}\widetilde{p}_k \varphi_k$ 
and $u_{\varrho h} = \sum_{k=1}^{N_h} u_k \varphi_k$ of the mixed finite
element scheme \eqref{FEM primal} and \eqref{FEM adjoint} as we have
already mentioned in Section~\ref{sec:FiniteElementErrorEstimates}.
Let $0 < \lambda_{\text{min}}(\mathbf{M}_h^{-1}\mathbf{K}_h) =
\lambda_1 \le \ldots \le \lambda_{N_h} =
\lambda_{\text{max}}(\mathbf{M}_h^{-1}\mathbf{K}_h)$ be the eigenvalues,
and $\mathbf{e}_1,\ldots,\mathbf{e}_{N_h} \in \mathbb{R}^{N_h}$ 
the corresponding eigenvectors of the generalized eigenvalue problem
\begin{equation}
\label{eqn:EVP:KM}
 \mathbf{K}_h \mathbf{e} = \lambda \, \mathbf{M}_h \mathbf{e},
\end{equation}
and let us suppose that the eigenvectors
$\mathbf{e}_1,\ldots,\mathbf{e}_{N_h}$ 
are orthonormal with respect to the $\mathbf{M}_h$-energy inner product 
$(\mathbf{M}_h \cdot,\cdot)$, i.e.,
$(\mathbf{M}_h \mathbf{e}_i,\mathbf{e}_j) = \delta_{i,j}$ for all
$i,j=1,\ldots,N_h$.
It is well known, see also \eqref{eqn:SpectralEquivalenceInequalitiesK}
and \eqref{eqn:SpectralEquivalenceInequalitiesM}, that there exists
positive constants $\underline{c}_\text{MK}$ and $\overline{c}_\text{MK}$,
which are independent of $h$, such that
\begin{equation}
\label{eqn:EV:KM}
\underline{c}_\text{MK} \le \lambda_{\text{min}}(\mathbf{M}_h^{-1}\mathbf{K}_h)
\quad \text{and} \quad
\lambda_{\text{max}}(\mathbf{M}_h^{-1}\mathbf{K}_h) \le \overline{c}_\text{MK} h^{-2}.
\end{equation}
These bounds are sharp with respect to $h$.

\begin{lemma}
\label{Lemma:SpectralEquivalenceS2M}
If $\varrho = h^4$, then the Schur complement $\mathbf{S}_h$ is spectrally
equivalent to the mass matrix $\mathbf{M}_h$. More precisely, there hold
the spectral equivalence inequalities
\begin{equation}
\label{eqn:SpectralEquivalenceInequalitiesMSM}
  \underline{c}_\text{MS} \, (\mathbf{M}_h\mathbf{v},\mathbf{v}) 
  \le (\mathbf{S}_h\mathbf{v},\mathbf{v}) \le 
  \overline{c}_\text{MS} \,(\mathbf{M}_h\mathbf{v},\mathbf{v}), 
  \quad \forall \mathbf{v} \in \mathbb{R}^{N_h},
\end{equation} 
where $\underline{c}_\text{MS} = 1$ and
$\overline{c}_\text{MS} = \overline{c}_\text{MK}^2 +1$,
with $\overline{c}_\text{MK}$ from \eqref{eqn:EV:KM}.
\end{lemma}

\begin{proof}
Let $\mathbf{v} \in \mathbb{R}^{N_h}$ be arbitrary, and let us expand 
$\mathbf{v} = \sum_{i=1}^{N_h} v_i^e \mathbf{e}_i $ in the orthonormal
eigenvector basis, where $v_i^e = (\mathbf{M}_h \mathbf{v},\mathbf{e}_i)$.
Then we can represent $(\mathbf{S}_h\mathbf{v},\mathbf{v})$ in the form
\begin{eqnarray*}
(\mathbf{S}_h\mathbf{v},\mathbf{v})
  &=& \sum_{i,j=1}^{N_h} v_i^e v_j^e (\varrho \,\mathbf{K}_h \mathbf{M}_h^{-1}
      \mathbf{K}_h\mathbf{e}_i + \mathbf{M}_h\mathbf{e}_i,\mathbf{e}_j)\\
  &=& \sum_{i,j=1}^{N_h} v_i^e v_j^e (\varrho \lambda_i^2 + 1)
      (\mathbf{M}_h\mathbf{e}_i,\mathbf{e}_j)
      \, = \, \sum_{i=1}^{N_h} (v_i^e)^2 (\varrho \lambda_i^2 + 1).
\end{eqnarray*} 
Using now the identity  $(\mathbf{M}_h\mathbf{v},\mathbf{v}) =
\sum_{i=1}^{N_h} (v_i^e)^2$, and the bounds \eqref{eqn:EV:KM}, we
immediately get the estimates
\begin{equation*}
 (\mathbf{S}_h\mathbf{v},\mathbf{v}) 
 \ge (\varrho \lambda_\text{min}(\mathbf{M}_h^{-1}\mathbf{K}_h)^2 + 1)
 \sum_{i=1}^{N_h} (v_i^e)^2
 \ge (\varrho \underline{c}_\text{MK}^2 + 1)
 (\mathbf{M}_h\mathbf{v},\mathbf{v})
 \ge (\mathbf{M}_h\mathbf{v},\mathbf{v})
\end{equation*}
and
\begin{equation*}
 (\mathbf{S}_h\mathbf{v},\mathbf{v}) 
 \le (\varrho \lambda_\text{max}(\mathbf{M}_h^{-1}\mathbf{K}_h)^2 + 1)
 \sum_{i=1}^{N_h} (v_i^e)^2
 \le (\varrho \overline{c}_\text{MK}^2 h^{-4}+ 1)
 (\mathbf{M}_h\mathbf{v},\mathbf{v}),
\end{equation*}
from which the spectral equivalence inequalities
\eqref{eqn:SpectralEquivalenceInequalitiesMSM} follow for $\varrho = h^4$.
\end{proof}

Since $\mathbf{S}_h =
\varrho \,\mathbf{K}_h \mathbf{M}_h^{-1} \mathbf{K}_h + \mathbf{M}_h$
is spectrally equivalent to the mass matrix $\mathbf{M}_h$, we can
efficiently solve the symmetric and positive definite Schur complement
system \eqref{eqn:Su=f} by means of the Preconditioned Conjugate
Gradient (PCG) method with the symmetric and positive definite
preconditioner $\mathbf{C}_h=\mathbf{M}_h$. Using the well-known
convergence estimate for the PCG method (see, e.g.,
\cite[Chapter~13]{LLSY:Steinbach:2008a}), and the spectral equivalence
inequalities \eqref{eqn:SpectralEquivalenceInequalitiesMSM} from
Lemma~\ref{Lemma:SpectralEquivalenceS2M}, we arrive at estimates of
the iteration error in the $L^2(\Omega)$-norm for the corresponding
finite element functions:
\begin{eqnarray*}
  \| u_{\varrho h} - u_{\varrho h}^n \|_{L^2(\Omega)} 
  &=& \| \mathbf{u} -  \mathbf{u}^n \|_{\mathbf{M}_h} 
      := (\mathbf{M}_h (\mathbf{u} -  \mathbf{u}^n),
      \mathbf{u} -  \mathbf{u}^n)^{1/2}    \\
  &\le& \underline{c}_{\text{MS}}^{-1/2} (\mathbf{S}_h (\mathbf{u} -
        \mathbf{u}^n),\mathbf{u} -  \mathbf{u}^n)^{1/2}\\
  &=& \underline{c}_{\text{MS}}^{-1/2} \| \mathbf{u} -
      \mathbf{u}^n \|_{\mathbf{S}_h} 
      \le \underline{c}_{\text{MS}}^{-1/2}
      \, 2 \, q^n \, \| \mathbf{u} -  \mathbf{u}^0 \|_{\mathbf{S}_h} \\
  &\le& \left(\frac{\overline{c}_{\text{MS}}}{\underline{c}_{\text{MS}}}
        \right)^{1/2} \, 2 \, q^n \,
        \| \mathbf{u} -  \mathbf{u}^0 \|_{\mathbf{M}_h}\\
  &=& \left(\frac{\overline{c}_{\text{MS}}}{\underline{c}_{\text{MS}}}
      \right)^{1/2} \, 2 \, q^n \,
      \| u_{\varrho h} - u_{\varrho h}^0\|_{L^2(\Omega)}, 
\end{eqnarray*} 
where
\begin{equation*}
  q = \frac{\sqrt{\text{cond}_2(\mathbf{M}_h^{-1}\mathbf{S}_h)}-1}
  {\sqrt{\text{cond}_2(\mathbf{M}_h^{-1}\mathbf{S}_h)}+1},
 \quad \text{with} \; \,
 \text{cond}_2(\mathbf{M}_h^{-1}\mathbf{S}_h) 
 = \frac{\lambda_\text{max}(\mathbf{M}_h^{-1}\mathbf{S}_h)}
 {\lambda_\text{min}(\mathbf{M}_h^{-1}\mathbf{S}_h)}
 \le \frac{ \overline{c}_{\text{MS}}}{ \underline{c}_{\text{MS}}}.
\end{equation*}
Since we can take $\underline{c}_{\text{MS}} = 1 $, we finally arrive at
the $L^2(\Omega)$ iteration error estimate  
\begin{equation}\label{eqn:PCG_IterationError_L_2}
  \| u_{\varrho h} - u_{\varrho h}^n \|_{L^2(\Omega)} 
      \le 2 \, (\overline{c}_{\text{MS}})^{1/2} \, \overline{q}^n \,
      \| u_{\varrho h} - u_{\varrho h}^0\|_{L^2(\Omega)},
\end{equation}
where
\[
  \overline{q} = \frac{\sqrt{\overline{c}_{\text{MS}}} - 1}
  {\sqrt{\overline{c}_{\text{MS}}} + 1} 
  = \frac{\sqrt{\overline{c}_\text{MK}^2 +1} - 1}
  {\sqrt{\overline{c}_\text{MK}^2 +1} + 1} 
  < 1
\]
is independent of $h$, using $\overline{c}_\text{MK}$ from \eqref{eqn:EV:KM}.
Here the finite element functions 
$u_{\varrho h}(x) = \sum_{k=1}^{N_h} u_k \varphi_k(x)$ and
$u_{\varrho h}^n(x) = \sum_{k=1}^{N_h} u_k^n \varphi_k(x)$ 
correspond to the solution
$\mathbf{u} = (u_k)_{k=1,\ldots,N_h} \in \mathbb{R}^{N_h}$ of the Schur
complement system \eqref{eqn:Su=f} and the $n$-th PCG iterate
$\mathbf{u}^n = (u_k^n)_{k=1,\ldots,N_h} \in \mathbb{R}^{N_h}$, respectively.

Now, using the triangle inequality, the $L^2(\Omega)$-norm discretization
error estimate \eqref{Final error L2 Hs}, and the $L^2(\Omega)$-norm
iteration error estimate \eqref{eqn:PCG_IterationError_L_2}, we finally get
\begin{eqnarray*}
  \| \overline{u} - u_{\varrho h}^n \|_{L^2(\Omega)}
  &\le& \| \overline{u} - u_{\varrho h} \|_{L^2(\Omega)} +
        \|  u_{\varrho h} - u_{\varrho h}^n \|_{L^2(\Omega)}\\
  &\le& c \, h^s \, \|\overline{u}\|_{H^s(\Omega)} +
        2 \, (\overline{c}_{\text{MS}})^{1/2} \, \overline{q}^n \,
        \| u_{\varrho h} - u_{\varrho h}^0\|_{L^2(\Omega)}   \\
  &\le& h^s \Big( c \,\|\overline{u}\|_{H^s(\Omega)} 
        + 2 \, (\overline{c}_{\text{MS}})^{1/2} \,
        \| u_{\varrho h} - u_{\varrho h}^0\|_{L^2(\Omega)}  \Big)
\end{eqnarray*}
provided that $\overline{q}^n \le h^s$ that yields
$ n \ge  \ln h^{-s} / \ln \overline{q}^{-1}$.
Since $\overline{q}$ is independent of $h$, the number of PCG iterations
only logarithmically grows with respect to $h$ in order to obtain the total 
error in the order $O(h^s)$ of the discretization error.
This logarithmical growth can be avoided in a nested iteration setting 
on a sequence of grids; see, e.g, \cite{LLSY:Hackbusch:2016a}.

However, each Schur Complement PCG iteration step requires the
solution of two systems of algebraic equations with the symmetric and
positive definite well-conditioned mass matrix $\mathbf{M}_h$ as
system matrix, namely,
\begin{enumerate}
 \item in the matrix-vector multiplication 
   $\mathbf{S}_h \mathbf{u}^n =
   \varrho \,\mathbf{K}_h \mathbf{M}_h^{-1} \mathbf{K}_h \mathbf{u}^n 
   + \mathbf{M}_h \mathbf{u}^n$,
 \item and in the preconditioning step
   $\mathbf{M}_h\mathbf{w}^n = \mathbf{r}^n$. 
\end{enumerate}
Thus, in a preprocessing step, we can factorize $\mathbf{M}_h$, e.g., by
means of the $\mathbf{L}_h \mathbf{D}_h \mathbf{L}_h^T$ or the Cholesky
factorization, and then use fast forward and backward substitutions at
each iteration step. In the preconditioning step, we can avoid the
solution of the preconditioning system
$\mathbf{C}_h\mathbf{w}^n = \mathbf{r}^n$ with $\mathbf{C}_h = \mathbf{M}_h$
by replacing $\mathbf{C}_h$ with a spectrally equivalent preconditioner
such as $\mathbf{C}_h = \text{diag}(\mathbf{M}_h)$, 
$\mathbf{C}_h = \text{lump}(\mathbf{M}_h) $,
$\mathbf{C}_h = \text{area}(\mathbf{M}_h) $,
or even $\mathbf{C}_h = h^d \mathbf{I}_h$;
cf. \eqref{eqn:SpectralEquivalenceInequalitiesM}. Here, 
$\text{diag}(\mathbf{M}_h)$ is the diagonal matrix with the diagonal
entries from $\mathbf{M}_h$, $\text{lump}(\mathbf{M}_h)$ is the lumped mass
matrix that is diagonal; see, e.g., \cite{LLSY:Thomee:1997a},
$\mathbf{C}_h = \text{area}(\mathbf{M}_h) $ denotes the diagonal matrix
with the $k$-th diagonal entry which coincides with the area of the
support of the basis function $\varphi_k$;
see \cite[Lemma~9.7]{LLSY:Steinbach:2008a} that also provides the
spectral equivalence constants, and $\mathbf{I}_h$ is the identity matrix.
If we replace $\mathbf{M}_h^{-1}$ in the Schur complement system
\eqref{eqn:Su=f} by the diagonal matrix $(\text{lump}(\mathbf{M}_h))^{-1}$,
then the solution of the corresponding inexact Schur complement system
\begin{equation}\label{eqn:InexactSchurComplementSystem}
  \Big[ \varrho \,\mathbf{K}_h (\text{lump}(\mathbf{M}_h))^{-1}
  \mathbf{K}_h + \mathbf{M}_h \Big] \widetilde{\mathbf{u}} =
  \mathbf{f} 
\end{equation}
by means of the PCG precondioned by $\text{diag}(\mathbf{M}_h)$,
$\text{lump}(\mathbf{M}_h)$ or $\mathbf{C}_h= \text{area}(\mathbf{M}_h)$, 
is obviously of asymptotically optimal
complexity $\mathcal{O}(N_h \ln \varepsilon^{-1})$ 
for some fixed relative accuracy $\varepsilon \in (0,1)$. 
However, the use of mass lumping in the discretization needs an
additional analysis that goes beyond the scope of this paper.
Instead, in Section~\ref{sec:NumericalResults}, 
we compare the discretization error 
and the efficiency of this inexact Schur comlement PCG, that we call
inexSCPCG, with the discretization error without mass lumping
corresponding to the solution of \eqref{coupled LGS}
and the efficiency of the other solvers proposed below. 

The solution of a mass matrix system in the matrix-vector multiplication
$\mathbf{S}_h \mathbf{u}^n$ can also be avoided by a mixed reformulation 
as a system of double size with a symmetric, but indefinite system matrix
that reads as follows: Find $(\mathbf{u}, \mathbf{\widehat{p}}) \in
\mathbb{R}^{N_h} \times \mathbb{R}^{N_h}$ such that
\begin{equation}
    \label{eqn:SymmetricInfiniteSystem} 
    \begin{bmatrix}
      \mathbf{M}_h & \mathbf{K}_h\\
      \mathbf{K}_h & -\varrho^{-1} \mathbf{M}_h \\
    \end{bmatrix}
    \begin{bmatrix}
      \mathbf{u}\\
      \mathbf{\widehat{p}}
    \end{bmatrix}
    =
    \begin{bmatrix}
      \mathbf{f}\\
      \mathbf{0}
    \end{bmatrix}
\end{equation}
where $\mathbf{\widehat{p}} =-\sqrt{\varrho}\,\mathbf{\widetilde{p}} =
- \mathbf{p}$. The symmetric and indefinite system
\eqref{eqn:SymmetricInfiniteSystem} is obviously equivalent to 
the non-symmetric and positive definite system \eqref{coupled LGS},
and, therefore, to the Schur complement system \eqref{eqn:Su=f}.

The block-diagonal matrix
\begin{equation}
\label{eqn:P_h}
    \mathcal{P}_h =
    \begin{bmatrix}
      \mathbf{M}_h + \varrho^{1/2}\mathbf{K}_h & \mathbf{0} \\
      \mathbf{0} &\varrho^{-1} (\mathbf{M}_h + \varrho^{1/2} \mathbf{K}_h) 
    \end{bmatrix}
\end{equation}
provides a preconditioner for the MINRES solver that is robust with respect
to the mesh size $h$ and the regularization parameter $\varrho$.
This result is proven in \cite{LSTY:Zulehner:2011a}, where also 
the corresponding convergence rate estimates are given.
In order to obtain an efficient, but at the same time robust preconditioner,
we have to replace the block-entry $\mathbf{M}_h + \varrho^{1/2}\mathbf{K}_h$
by a spectrally equivalent preconditioner $\mathbf{C}_h$ such that 
the spectral equivalence constants do not depend neither on $h$ nor on
$\varrho$. Symmetric and positive definite multigrid preconditioners 
based on symmetric $V$- or $W$-cycles are certainly suitable candidates.
We came back to this choice later.

We recall that the optimal choice of the regularization parameter is
$\varrho = h^4$. Using again the eigenvector expansion
$\mathbf{v} = \sum_{i=1}^{N_h} v_i^e \mathbf{e}_i $, we get
\begin{equation}\label{eqn:M+K}
  ((\mathbf{M}_h + \varrho^{1/2}\mathbf{K}_h) \mathbf{v},\mathbf{v}) 
  = \sum_{i=1}^{N_h} (1 + \varrho^{1/2} \lambda_i) (v_i^e)^2 
  = \sum_{i=1}^{N_h} (1 + h^2 \lambda_i) (v_i^e)^2 
\end{equation}
for all $\mathbf{v} \in \mathbb{R}^{N_h}$. This representation and the
eigenvalue estimates \eqref{eqn:EV:KM} immediately yield the spectral
equivalence inequalities
\begin{equation*}
\label{eqn:SpectralEquivalenceInequalitiesMAM1}
  (1 + h^2 \underline{c}_\text{MK}) \, (\mathbf{M}_h\mathbf{v},\mathbf{v}) 
  \le ((\mathbf{M}_h + \varrho^{1/2}\mathbf{K}_h)\mathbf{v},\mathbf{v}) \le 
  (1 + \overline{c}_\text{MK}) \,(\mathbf{M}_h\mathbf{v},\mathbf{v}), 
  \; \forall \mathbf{v} \in \mathbb{R}^{N_h},
\end{equation*} 
which we now rewrite in short form as
\begin{equation}
\label{eqn:SpectralEquivalenceInequalitiesMAM2}
 \underline{c}_\text{MA}\,\mathbf{M}_h
  \le \mathbf{A_h} \le 
  \overline{c}_\text{MA} \,\mathbf{M}_h, 
\end{equation} 
where $\mathbf{A_h} = \mathbf{M}_h + \varrho^{1/2}\mathbf{K}_h
= \mathbf{M}_h + h^2 \mathbf{K}_h$,
$\underline{c}_\text{MA} =  1 + h^2 \underline{c}_\text{MK} \ge 1$, and
$\overline{c}_\text{MA} = 1 + \overline{c}_\text{MK}$.
Replacing now the mass matrix $\mathbf{M}_h$ by the diagonal 
preconditioner $\mathbf{C}_h = \text{diag}(\mathbf{M}_h)$, we 
see that 
\begin{equation}
\label{eqn:SpectralEquivalenceInequalitiesCAC}
  \underline{c}_\text{CA}\,\mathbf{C}_h
  \le \mathbf{A_h} \le 
  \overline{c}_\text{CA} \,\mathbf{C}_h, 
\end{equation} 
with $h$-independent, positive constants $\underline{c}_\text{CA}$ 
and $\overline{c}_\text{CA}$. These spectral equivalence inequalities
easily follow from \eqref{eqn:SpectralEquivalenceInequalitiesMAM2}
and the spectral equivalence of $\mathbf{M}_h$ to
$ \text{diag}(\mathbf{M}_h)$. Therefore, the diagonal preconditioner
\begin{equation}
\label{eqn:P_{diag,h}}
    \mathcal{P}_{\text{diag},h} =
    \begin{bmatrix}
      \mathbf{C}_h &  \mathbf{0}\\
      \mathbf{0}   &  h^{-4} \mathbf{C}_h 
    \end{bmatrix}
    =
    \begin{bmatrix}
      \text{diag}(\mathbf{M}_h) &  \mathbf{0}\\
      \mathbf{0}   &  h^{-4} \text{diag}(\mathbf{M}_h)
    \end{bmatrix}
\end{equation}
is spectrally equivalent to ${\mathcal P}_h$.
Thus, the MINRES preconditioned by $\mathcal{P}_{\text{diag},h}$ 
converges with a rate that is independent of $h$,
and with an asymptotically 
optimal complexity $\mathcal{O}(N_h \ln \varepsilon^{-1})$
for some fixed relative accuracy $\varepsilon \in (0,1)$.
We call this solver 
$\mathcal{P}_{\text{diag}}\mbox{MINRES}$. Numerical tests underpinning
our theoretical results are presented in Section~\ref{sec:NumericalResults}.
If we replace $\mathbf{M}_h + \varrho^{1/2}\mathbf{K}_h$ by 
$\mathbf{C}_h= \text{lump}(\mathbf{M}_h)$, 
$\mathbf{C}_h= \text{area}(\mathbf{M}_h)$, 
or 
$\mathbf{C}_h= \text{diag}(\mathbf{M}_h + h^2 \mathbf{K}_h)$,
we also get diagonal preconditioners with the same properties
as $\text{diag}(\mathbf{M}_h)$.
 
At first glance, it does not seem necessary to replace
$\mathbf{M}_h + \varrho^{1/2}\mathbf{K}_h$ by a symetric and positive
definite multigrid preconditioner 
$\mathbf{C}_h= (\mathbf{M}_h + \varrho^{1/2}\mathbf{K}_h)
(\mathbf{I}_h - \mathbf{MG}_h^k)^{-1}$ as it was proposed in
\cite{LSTY:Zulehner:2011a}  for fixed positive $\varrho$, since, for
$\varrho = h^4$, the block $\mathbf{M}_h + \varrho^{1/2}\mathbf{K}_h$ 
is spectrally equivalent to $\mathbf{M}_h$ and, therefore, well conditioned.
Here, $\mathbf{I}_h$ is the identity, $\mathbf{MG}$ denotes the multigrid
iteration matrix, and $k$ is the number of multigrid iterations per
precondioning step (usually, $k=1$). However, we can improve the spectral
equivalent constants in comparison to
$\mathbf{C}_h = \text{diag}(\mathbf{M}_h)$ for the price of higher
arithmetical, but still asymptotically optimal cost per preconditioning step.
More precisely, for the multigrid preconditioner
$\mathbf{C}_h = \mathbf{A}_h(\mathbf{I}_h - \mathbf{MG}_h^k)^{-1}$,
we get the spectral equivalence inequalities
\eqref{eqn:SpectralEquivalenceInequalitiesCAC} with the spectral constants 
$\underline{c}_\text{CA} = 1 - \eta^k$ and $\overline{c}_\text{CA} = 1 + \eta^k$,
where $\eta$ is an upper bound for the $\mathbf{A}_h$-energy norm of the 
multigrid iteration matrix $\mathbf{MG}_h$ that is nothing but the
multigrid convergence rate in the $\mathbf{A}_h$-energy norm. 
We mention that $\overline{c}_\text{CA}$ is even equal to $1$ if $k$ is
even or if the multigrid iteration matrix $\mathbf{MG}_h$ is non-negative
with respect to the $\mathbf{A}_h$-energy inner product. We refer the reader
to \cite{LLSY:JungLanger:1991a, LLST:JungLangerMeyerQueckSchneider:1989a}
for details on multigrid preconditioners. Therefore, the multigrid
preconditioner
\begin{equation}
\label{eqn:P_{mg,h}}
    \mathcal{P}_{\text{mg},h} =
    \begin{bmatrix}
      \mathbf{C}_h &  \mathbf{0}\\
      \mathbf{0}   &  h^{-4} \mathbf{C}_h 
    \end{bmatrix}
    =
    \begin{bmatrix}
      \mathbf{A}_h(\mathbf{I}_h - \mathbf{MG}_h^k)^{-1} &  \mathbf{0}\\
      \mathbf{0}   &  h^{-4} \mathbf{A}_h(\mathbf{I}_h - \mathbf{MG}_h^k)^{-1}
    \end{bmatrix}
\end{equation}
is spectrally equivalent to ${\mathcal P}_h$. Thus, the MINRES preconditioned
by $\mathcal{P}_{\text{mg},h}$ also converges with a rate that is independent of $h$, 
and with asymptotically optimal complexity $\mathcal{O}(N_h \ln \varepsilon^{-1})$.
We call this solver ${\mathcal{P}_{\text{mg}}\mbox{MINRES}}$.
Numerical tests illustrating the behavior of this multigrid preconditioner 
and comparing it with the diagonal preconditioner $\mathcal{P}_{\text{diag},h}$
are reported in Section~\ref{sec:NumericalResults}.

Instead of the MINRES, we can also use Bramble-Pasciak PCG ({BP-PCG}) as
solver; see \cite{LLSY:BramblePasciak:1988a}. In order to apply BP-PCG, we
reformulate the coupled system (\ref{coupled LGS}) in the equivalent form
\begin{equation}
    \label{eqn:SymmetricInfiniteSystem2BP} 
    \begin{bmatrix}
      \mathbf{M}_h & \sqrt{\varrho}\, \mathbf{K}_h\\
      \sqrt{\varrho}\,\mathbf{K}_h & - \mathbf{M}_h \\
    \end{bmatrix}
    \begin{bmatrix}
      \mathbf{\widetilde{p}}\\
      \mathbf{u}
    \end{bmatrix}
    =
    \begin{bmatrix}
      \mathbf{0}\\
      \mathbf{-f}
    \end{bmatrix}
\end{equation}
Applying the Bramble-Pasciak transformation 
\begin{equation*}
    \mathcal{T}_h=
    \begin{bmatrix}
      \mathbf{M}_h \mathbf{C}_{M_h}^{-1}-\mathbf{I}                 &  \mathbf{0} \\
      \sqrt{\varrho}\, \mathbf{K}_h \mathbf{C}_{M_h}^{-1}     &  - \mathbf{I}\\
    \end{bmatrix}
    =
    \begin{bmatrix}
      (\mathbf{M}_h - \mathbf{C}_{M_h})   \mathbf{C}_{M_h}^{-1}\mathbf{I}                 &  \mathbf{0} \\
      \sqrt{\varrho}\, \mathbf{K}_h \mathbf{C}_{M_h}^{-1}     &  - \mathbf{I}\\
    \end{bmatrix}
\end{equation*}
to the symmetric, but indefinite system \eqref{eqn:SymmetricInfiniteSystem2BP},
we arrive at the symetric and positive definite system 
\begin{equation}
\label{eqn:spdBPsystem}
  \mathcal{K}_h 
      \begin{bmatrix}
      \mathbf{\widetilde{p}}\\
      \mathbf{u}
      \end{bmatrix} 
      = 
      \begin{bmatrix}
      \mathbf{0}\\
      \mathbf{f}
     \end{bmatrix}
      \equiv
      \mathcal{T}_h
      \begin{bmatrix}
      \mathbf{0}\\
      \mathbf{-f}
     \end{bmatrix},
\end{equation}
with the system matrix
\begin{equation*}
    \begin{aligned}
    {\mathcal K}_h&=
    \begin{bmatrix}
     \mathbf{M}_h \mathbf{C}_{M_h}^{-1}-\mathbf{I}                 &  \mathbf{0} \\
      \sqrt{\varrho}\, \mathbf{K}_h \mathbf{C}_{M_h}^{-1}     &  - \mathbf{I}\\
    \end{bmatrix}
    \begin{bmatrix}
     \mathbf{M}_h & \sqrt{\varrho}\, \mathbf{K}_h\\
      \sqrt{\varrho}\,\mathbf{K}_h & - \mathbf{M}_h \\
    \end{bmatrix}\\
    &=
    \begin{bmatrix}
      \mathbf{M}_h \mathbf{C}_{M_h}^{-1}\mathbf{M}_h-\mathbf{M}_h 
                                        & \sqrt{\varrho}(\mathbf{M}_h\mathbf{C}_{M_h}^{-1}-\mathbf{I})\mathbf{K}_h \\
      \sqrt{\varrho} \mathbf{K}_h(\mathbf{C}_{M_h}^{-1}\mathbf{M}_h-\mathbf{I}) 
                                        & \varrho \mathbf{K}_h \mathbf{C}_{M_h}^{-1} \mathbf{K}_h + \mathbf{M}_h\\
    \end{bmatrix}\\
    &=
    \begin{bmatrix}
      (\mathbf{M}_h - \mathbf{C}_{M_h})\mathbf{C}_{M_h}^{-1}\mathbf{M}_h
                                        & \sqrt{\varrho}(\mathbf{M}_h-\mathbf{C}_{M_h})\mathbf{C}_{M_h}^{-1}\mathbf{K}_h \\
      \sqrt{\varrho} \mathbf{K}_h(\mathbf{C}_{M_h}^{-1}\mathbf{M}_h-\mathbf{I}) 
                                        & \varrho \mathbf{K}_h \mathbf{C}_{M_h}^{-1} \mathbf{K}_h + \mathbf{M}_h\\
    \end{bmatrix},
    \end{aligned}
\end{equation*}
where $\mathbf{C}_{M_h}$ is symmetric and positive definite, spectrally
equivalent to the mass matrix $\mathbf{M}_h$, and 
\begin{equation}
\label{eqn:C<M}
 \mathbf{C}_{M_h} < \mathbf{M}_h.
\end{equation}
Therefore, we can choose 
\begin{equation}
\label{eqn:C=0.25M}
\mathbf{C}_{M_h} = 0.25 \, \text{diag}(\mathbf{M}_h)
\end{equation}
that is spectrally equivalent to $\mathbf{M}_h$,
and that ensures \eqref{eqn:C<M}; see, e.g., \cite{Elman2005}. 
More precisely, there are spectral equivalence constants
$1 < \underline{c}_\text{CM} \le \overline{c}_\text{CM}$ such that
\[
  \mathbf{C}_{M_h}   < \underline{c}_\text{CM} \mathbf{C}_{M_h} \le
  \mathbf{M}_h \le \overline{c}_\text{CM}\mathbf{C}_{M_h}
\]
Then the exact BP preconditioner
\begin{equation}
\label{eqn:exactBPpreconditioner}
    {\mathcal P}_{BP,h} =
    \begin{bmatrix}
      \mathbf{M}_h - \mathbf{C}_{M_h} & \mathbf{0} \\
      \mathbf{0} &\varrho \mathbf{K}_h \mathbf{M}_{h}^{-1} \mathbf{K}_h + \mathbf{M}_h 
    \end{bmatrix}
\end{equation}
is spectrally equivalent to ${\mathcal K}_h$. More precisely, the spectral
equivalence inequalities
\begin{equation}
\label{eqn:SpectralEquivalenceInequalitiesBPexact}
\underline{c}_\text{PK} \le {\mathcal P}_{BP,h} \le{\mathcal K}_h \le \overline{c}_\text{PK}{\mathcal P}_{BP,h}
\end{equation}
hold with the spectral equivalence constants
\begin{equation}
\label{eqn:SpectralEquivalenceConstantsBPexact}
\underline{c}_\text{PK} = \frac{1 - \sqrt{\alpha}}{1 - \alpha} 
\quad \mbox{and} \quad 
\overline{c}_\text{PK} = \frac{1 + \sqrt{\alpha}}{1 - \alpha}, 
\end{equation}
where $\alpha = 1 - (1/\underline{c}_\text{CM})$. The lower constant
$\underline{c}_\text{PK}$ was derived in \cite{LLSY:BramblePasciak:1988a},
wheras the upper constant can be found in \cite{LLSY:Zulehner:2001a}. 
Now, replacing the Schur complement
$ \mathbf{S}_h = \varrho \mathbf{K}_h \mathbf{M}_{h}^{-1} \mathbf{K}_h +
\mathbf{M}_h$ in the exact BP preconditioner
\eqref{eqn:exactBPpreconditioner} by $\text{diag}(\mathbf{M}_h)$
that is spectrally equivalent to $ \mathbf{S}_h$, 
we arrive at the inexact BP preconditioner
\begin{equation*}
    \widetilde{\mathcal P}_{BP,h} =
    \begin{bmatrix}
      \mathbf{M}_h - \mathbf{C}_{M_h} & \mathbf{0} \\
      \mathbf{0} & \text{diag}(\mathbf{M}_h)
    \end{bmatrix}
    =
    \begin{bmatrix}
      \mathbf{M}_h - 0.25 \, \text{diag}(\mathbf{M}_h) & \mathbf{0} \\
      \mathbf{0} & \text{diag}(\mathbf{M}_h)
    \end{bmatrix}
\end{equation*}
that is spectrally equivalent to ${\mathcal K}_h$ as well.
Thus, the BP-PCG, that is here nothing but the PCG preconditioned by
$\widetilde{\mathcal P}_{BP,h}$  
applied to the symmetric and positive definite system
\eqref{eqn:SymmetricInfiniteSystem2BP},
converges with a $h$-independent rate in asymptotically 
optimal complexity $\mathcal{O}(N_h \ln \varepsilon^{-1})$. 
In the next section, we numerically compare exactly this BP-PCG
with ${\mathcal{P}_{\text{diag}}\mbox{MINRES}}$
and ${\mathcal{P}_{\text{mg}}\mbox{MINRES}}$ 
as well as with the inexact Schur complement PCG inexactSCPCG
where we use mass lumping in the discretization in order to make the
multiplications with the Schur complement effficient.

%
%

\section{Numerical results}
\label{sec:NumericalResults}
In our numerical examples, we 
consider
the computational domain $\Omega=(0,1)^3$,
that is decomposed into uniformly refined tetrahedral elements. The starting
mesh contains $384$ tetrahedral elements and $125$ vertices, i.e., $5$
vertices in each direction, which leads to an initial mesh size
$h=2^{-2}$. The tests are performed on $8$ uniformly refined mesh levels
$L_i$, $i=1,...,8$. The number of vertices, the mesh size $h$,
and the corresponding regularization parameter $\varrho=h^4$ are
given in Table \ref{tab:hrho}. 
\begin{table}[h]
\begin{tabular}{|l|rlr|}
  \hline
  Level&Number of vertices&$h$&$\varrho\; (=h^4)$\\
  \hline
  $L_1$&$125$&$2^{-2}$&$2^{-8}$\\
  $L_2$&$729$&$2^{-3}$&$2^{-12}$\\
  $L_3$&$4,913$&$2^{-4}$&$2^{-16}$\\
  $L_4$&$35,937$&$2^{-5}$&$2^{-20}$\\
  $L_5$&$274,625$&$2^{-6}$&$2^{-24}$\\
  $L_6$&$2,146,689$&$2^{-7}$&$2^{-28}$\\
  $L_7$&$16,974,593$&$2^{-8}$&$2^{-32}$\\
  $L_8$&$135,005,697$&$2^{-9}$&$2^{-36}$\\
  \hline
\end{tabular}
\caption{The number of vertices, the mesh size $h$, and the
related regularization parameter $\varrho=h^4$ on $8$ uniformly refined
mesh levels.}
\label{tab:hrho}
\end{table}

To confirm the convergence rate as given in \eqref{Final error L2 H2}
of the finite element solution $u_{\varrho h}$ to a given
target $\overline{u}$, we have considered the following four representative
targets with different regularities, similar to
\cite{LLSY:NeumuellerSteinbach:2021a}:

\medskip

\noindent
{\bf Target 1:} A smooth target $\overline{u}=\sin(\pi x)\sin(\pi y)\sin( \pi
z)$, $\overline{u}\in H_0^1(\Omega)\cap H^2(\Omega)$;

\medskip

\noindent
{\bf Target 2:} A piecewise linear continuous target $\overline{u}$ being
one in the mid point $(\frac{1}{2},\frac{1}{2},\frac{1}{2})$ and zero in all
corner points of $\Omega$, $\overline{u}\in H_0^1(\Omega)\cap
H^s(\Omega)$, $ s < \frac{3}{2}$;

\medskip

\noindent
{\bf Target 3:} A piecewise constant discontinuous function $\overline{u}$
being one in the inscribed cube $(\frac{1}{4}, \frac{3}{4})^3$ and zero
elsewhere, $\overline{u}\in H^s(\Omega)$, $s < \frac{1}{2}$;

\medskip

\noindent
{\bf Target 4:} A smooth target
$\overline{u}=1+\sin(\pi x_1)\sin(\pi x_2)\sin(\pi x_3)$
that violates the homogeneous Dirichlet boundary conditions,
$\overline{u}\in H^s(\Omega)$, $s<\frac{1}{2}$. 

\medskip

We will further report robustness and computational cost of four preconditioned
Krylov subspace solvers for the large scale linear system of algebraic
equations that are arising from the finite element discretization of the
optimality system with the choice of the regularization
parameter $\varrho=h^4$. More precisely, we study the numerical performance
of the following four Krylov subspace solvers
described in Section~\ref{sec:RobustSolvers}:

\begin{enumerate}
\item ${\mathcal{P}_{\text{mg}}\mbox{MINRES}}$: multigrid-preconditioned
  MINRES for solving \eqref{eqn:SymmetricInfiniteSystem},
\item ${\mathcal{P}_{\text{diag}}\mbox{MINRES}}$: diagonal-preconditioned
  MINRES for solving \eqref{eqn:SymmetricInfiniteSystem},
\item BP-PCG: Bramble-Pasciak PCG for solving
  \eqref{eqn:SymmetricInfiniteSystem2BP},
\item inexSCPCG: inexact Schur complement PCG solving
  \eqref{eqn:InexactSchurComplementSystem}.
\end{enumerate}
The multigrid preconditioner $\eqref{eqn:P_{mg,h}}$ with $k=1$, 
which is applied  in ${\mathcal{P}_{\text{mg}}\mbox{MINRES}}$,
is based on a W-cycle that
starts with a zero initial guess, uses 2 forward Gauss--Seidel presmoothing 
and 2 backward Gauss--Seidel postsmoothing steps, and canonical
transfer operators such that the multigrid preconditioner is
symmetric and positive definite;
see \cite{LLSY:JungLanger:1991a, LLST:JungLangerMeyerQueckSchneider:1989a}.
We note that the first three solvers solve mixed systems with $2 N_h$
degrees of freedoms, whereas the last method solves the inexact Schur
complement system that only has $N_h$ degrees of freedom.
We recall that the systems \eqref{coupled LGS},
\eqref{eqn:SymmetricInfiniteSystem},
\eqref{eqn:SymmetricInfiniteSystem2BP} and \eqref{eqn:Su=f} are equivalent.
But since we use the inexact Schur complement instead of the exact Schur
  complement in \eqref{eqn:InexactSchurComplementSystem}, we compute a
  perturbed solution $\widetilde{u}_{\varrho h}$ with some additional
  error. Hence we also compare all discretization errors.
In all of these approaches, the solvers stop the iterations as soon as the
preconditioned residual is reduced by a factor $10^{11}$.
Since the residual of the inexact Schur complement system
  \eqref{eqn:InexactSchurComplementSystem} is computed for the primal
  unknown $\mathbf{u}$ only, the resulting $L^2$ error of
  $\widetilde{u}_{\varrho h}$ is different compared to the system
  \eqref{coupled LGS}, where also the residual of the adjoint
  $\mathbf{p}$ is involved.
We finally mention that the preconditioned residual norm here
reproduces the $L^2$ norm 
  in which we are primarily interested.


\subsection{Convergence studies}
The errors $\|u_{\varrho h}-\overline{u}\|_{L^2(\Omega)}$ between the finite
element solution $u_{\varrho h}$ and the given target $\overline{u}$ are
computed by means of the first three methods that solve the equivalent mixed 
formulations \eqref{eqn:SymmetricInfiniteSystem} or
\eqref{eqn:SymmetricInfiniteSystem2BP}, whereas the errors
$\|\widetilde{u}_{\varrho h}-\overline{u}\|_{L^2(\Omega)}$ are computed 
by solving the inexact Schur complement equation
\eqref{eqn:InexactSchurComplementSystem} using inexSCPCG.
These errors are given in Tables \ref{tab:eocexample1}--\ref{tab:eocexample4}
for Targets 1--4, respectively. For all of these cases, we have used the
required scaling $\varrho=h^4$, which leads to optimal convergence with
respect to the mesh size $h$, depending on the corresponding regularity of
the given target $\overline{u}$. This is observed as the experimental order
of convergence (eoc) in Tables \ref{tab:eocexample1}--\ref{tab:eocexample4}.
Further, the solution $\widetilde{u}_{\varrho h}$ from solving the inexact
Schur complement equation (Approach 4) does not deteriorate with
respect to the accuracy and convergence rate. This is confirmed by comparison
with the solution from the exact Schur complement equation that is equivalent
to solving the mixed fromulations \eqref{eqn:SymmetricInfiniteSystem}
or \eqref{eqn:SymmetricInfiniteSystem2BP} 
as we did in the first three approaches.
  
  \begin{table}[h]
    \begin{tabular}{|l|lr|lr|}
      \hline
      \multirow{2}{*}{Level}&\multicolumn{2}{c|}{Approaches 1 - 3} &
      \multicolumn{2}{c|}{Approach 4} \\  \cline{2-5}
      &$\|u_{\varrho h}-\overline{u}\|_{L^2(\Omega)}$ &eoc& $\|\widetilde{u}_{\varrho h}-\overline{u}\|_{L^2(\Omega)}$ & eoc\\
      \hline
      $L_1$ &$3.04904$e$-1$& $-$& $3.03162$e$-1$& $-$\\
      $L_2$ &$7.14457$e$-2$& $2.09$&$6.92534$e$-2$& $2.13$\\
      $L_3$ &$5.35113$e$-3$& $3.74$&$5.29228$e$-3$& $3.71$\\
      $L_4$ &$6.22449$e$-4$& $3.10$&$6.19849$e$-4$& $3.09$\\
      $L_5$ &$1.34331$e$-4$& $2.21$&$1.33758$e$-4$& $2.21$\\
      $L_6$ &$3.27079$e$-5$& $2.03$&$3.25740$e$-5$& $2.04$\\
      $L_7$ &$8.07438$e$-6$& $2.02$&$8.04282$e$-6$& $2.02$\\
      $L_8$ &$2.00173$e$-6$& $2.01$&$1.99422$e$-6$& $2.01$\\ \hline
      Theory: & \multicolumn{2}{r|}{$2$} & \multicolumn{2}{r|}{$2$} \\
      \hline
    \end{tabular}
    \caption{Comparison of error $\|u_{\varrho
        h}-\overline{u}\|_{L^2(\Omega)}$ (Approaches 1 - 3) and $\|\widetilde{u}_{\varrho
        h}-\overline{u}\|_{L^2(\Omega)}$ (Approach 4) for 
      {\bf Target 1}.
    }
    \label{tab:eocexample1}
  \end{table}

  \begin{table}[h]
    \begin{tabular}{|l|lr|lr|}
      \hline
      \multirow{2}{*}{Level}&\multicolumn{2}{c|}{Approaches 1 - 3} &
      \multicolumn{2}{c|}{Approach 4} \\  \cline{2-5}
      &$\|u_{\varrho h}-\overline{u}\|_{L^2(\Omega)}$ &eoc& $\|\widetilde{u}_{\varrho h}-\overline{u}\|_{L^2(\Omega)}$ & eoc\\
      \hline
      $L_1$ &$2.72445$e$-1$ & $-$&$2.71300$e$-1$ & $-$\\
      $L_2$&$8.50409$e$-2$&$1.68$&$8.41925$e$-2$& $1.69$\\
      $L_3$&$2.99226$e$-2$& $1.51$&$2.90354$e$-2$& $1.54$\\
      $L_4$&$1.04906$e$-2$& $1.51$&$1.00864$e$-2$& $1.53$\\
      $L_5$&$3.70527$e$-3$& $1.50$&$3.54103$e$-3$& $1.51$\\
      $L_6$&$1.30970$e$-3$ & $1.50$&$1.24752$e$-3$&$1.51$\\
      $L_7$&$4.63061$e$-4$& $1.50$&$4.40293$e$-4$& $1.50$\\
      $L_8$&$1.63735$e$-4$& $1.50$&$1.55529$e$-4$& $1.50$\\ \hline
      Theory: & \multicolumn{2}{r|}{$1.5$} & \multicolumn{2}{r|}{$1.5$} \\
      \hline
    \end{tabular}
    \caption{Comparison of error $\|u_{\varrho
        h}-\overline{u}\|_{L^2(\Omega)}$ (Approaches 1 - 3) and $\|\widetilde{u}_{\varrho
        h}-\overline{u}\|_{L^2(\Omega)}$ (Approach 4) for 
      {\bf Target 2}.
    }\label{tab:eocexample2}
  \end{table}

  \begin{table}[h]
  \begin{tabular}{|l|lr|lr|}
    \hline
      \multirow{2}{*}{Level}&\multicolumn{2}{c|}{Approaches 1 - 3} &
      \multicolumn{2}{c|}{Approach 4} \\  \cline{2-5}
      &$\|u_{\varrho h}-\overline{u}\|_{L^2(\Omega)}$ &eoc& $\|\widetilde{u}_{\varrho h}-\overline{u}\|_{L^2(\Omega)}$ & eoc\\
      \hline
      $L_1$ &$3.28255$e$-1$ & $-$&$3.26425$e$-1$& $-$\\
      $L_2$ &$2.30561$e$-1$& $0.51$&$2.25595$e$-1$& $0.53$\\
      $L_3$ & $1.63827$e$-1$& $0.49$&$1.59922$e$-1$&$0.50$\\
      $L_4$ &$1.15682$e$-1$& $0.50$&$1.12852$e$-1$& $0.50$\\
      $L_5$&$8.16986$e$-2$& $0.50$&$7.96806$e$-2$& $0.50$\\
      $L_6$&$5.77276$e$-2$ & $0.50$&$5.62946$e$-2$ & $0.50$\\
      $L_7$&$4.08035$e$-2$& $0.50$&$3.97882$e$-2$& $0.50$\\
      $L_8$&$2.88466$e$-2$& $0.50$&$2.81281$e$-2$& $0.50$\\ \hline
      Theory: & \multicolumn{2}{r|}{$0.5$} & \multicolumn{2}{r|}{$0.5$} \\
      \hline
    \end{tabular}
    \caption{Comparison of error $\|u_{\varrho
        h}-\overline{u}\|_{L^2(\Omega)}$ (Approaches 1 - 3) and $\|\widetilde{u}_{\varrho
        h}-\overline{u}\|_{L^2(\Omega)}$ (Approach 4) for 
      {\bf Target 3}.
    }\label{tab:eocexample3}
  \end{table}

  \begin{table}[h]
  \begin{tabular}{|l|lr|lr|}
    \hline
      \multirow{2}{*}{Level}&\multicolumn{2}{c|}{Approaches 1 - 3} &
      \multicolumn{2}{c|}{Approach 4} \\  \cline{2-5}
      &$\|u_{\varrho h}-\overline{u}\|_{L^2(\Omega)}$ &eoc& $\|\widetilde{u}_{\varrho h}-\overline{u}\|_{L^2(\Omega)}$ & eoc\\
      \hline
      $L_1$ & $1.15861$e$-0$& $-$&$1.15659$e$-0$&$-$\\
      $L_2$&$6.72524$e$-1$&$0.78$&$6.73325$e$-1$&$0.78$\\
      $L_3$&$4.63819$e$-1$&$0.54$&$4.62241$e$-1$&$0.54$\\
      $L_4$&$3.27310$e$-1$&$0.50$&$3.25524$e$-1$&$0.51$\\
      $L_5$&$2.31129$e$-1$&$0.50$&$2.29647$e$-1$&$0.50$\\
      $L_6$&$1.63305$e$-1$&$0.50$&$1.62176$e$-1$&$0.50$\\
      $L_7$&$1.15426$e$-1$& $0.50$&$1.14599$e$-1$& $0.50$\\
      $L_8$&$8.16011$e$-2$& $0.50$&$8.10057$e$-2$& $0.50$\\ \hline
      Theory: & \multicolumn{2}{r|}{$0.5$} & \multicolumn{2}{r|}{$0.5$} \\
      \hline
    \end{tabular}
    \caption{Comparison of error $\|u_{\varrho h}-\overline{u}\|_{L^2(\Omega)}$ (Approaches 1 - 3) 
    and $\|\widetilde{u}_{\varrho
      h}-\overline{u}\|_{L^2(\Omega)}$ (Approach 4) for 
    {\bf Target 4}.
    }\label{tab:eocexample4}
  \end{table}
  
  
\subsection{Solver performance}
We recall that all solvers stop the iterations as soon as the preconditioned
residual is reduced by a factor $10^{11}$.
A comparison of the number of iterations (\#Its)
and the required solving time in seconds (s) using these four solvers
are provided in Tables \ref{tab:costexample1}--\ref{tab:costexample4} for the
given targets {\bf Target 1--4}, respectively. 
${\mathcal{P}_{\text{mg}}\mbox{MINRES}}$ requires the fewest iteration numbers
among all these solvers. The solver inexSCPCG outperforms the other three 
solvers regarding the solving time. This is mainly due to the fact that
the inexact Schur complement equation only needs half of degrees of freedom in
comparison with the mixed formulations. Finally, we observe that all of the
preconditioned Krylov subspace methods show their robustness with respect to
the mesh size $h$, using the particular choice for the regularization
parameter $\varrho=h^4$. Furthermore, solvers 
${\mathcal{P}_{\text{diag}}\mbox{MINRES}}$, BP-PCG, and inexSCPCG
are relatively easy to parallelize due to the fact that each iteration of
these approaches only requires matrix-vector multiplications, and the
preconditioning step only requires a vector scaling operation by simply
using the diagonal of the corresponding matrix as a preconditioner thanks
to the spectral equivalence inequalities. 
  \begin{table}[h]
    \begin{tabular}{|l|rl|rl|rl|rl|}
      \hline
      \multirow{2}{*}{Level}&\multicolumn{2}{c|}{${\mathcal{P}_{\text{mg}}\mbox{MINRES}}$} &
      \multicolumn{2}{c|}{${\mathcal{P}_{\text{diag}}\mbox{MINRES}}$} &
      \multicolumn{2}{c|}{BP-PCG} &
      \multicolumn{2}{c|}{inexSCPCG} \\  \cline{2-9}
      &\#Its&Time (s) &\#Its&Time (s) &\#Its&Time (s) &\#Its&Time (s)   \\
      \hline
      $L_1$ &$19$&$2.9$e$-2$&$21$&$2.9$e$-3$ &$24$&$4.9$e$-2$ &$10$&$6.6$e$-4$  \\
      $L_2$ &$21$&$2.7$e$-1$&$172$&$1.4$e$-1$&$180$&$4.5$e$-1$&$88$&$2.5$e$-2$  \\
      $L_3$ &$20$&$1.8$e$-0$&$234$&$2.3$e$-0$&$254$&$1.0$e$-0$&$126$&$2.6$e$-1$ \\
      $L_4$ &$20$&$1.4$e$+1$&$231$&$3.8$e$+1$&$247$&$1.3$e$+1$&$132$&$3.3$e$-0$ \\
      $L_5$ &$18$&$1.0$e$+2$&$225$&$4.2$e$+2$&$242$&$2.0$e$+2$&$130$&$2.0$e$+1$ \\
      $L_6$ &$18$&$8.1$e$+2$&$220$&$2.1$e$+3$&$235$&$2.6$e$+3$&$128$&$1.8$e$+2$ \\
      $L_7$ &$18$&$7.9$e$+3$&$213$&$2.4$e$+4$&$229$&$9.5$e$+3$&$124$&$4.1$e$+3$ \\
      $L_8$ &$18$&$7.1$e$+4$&$205$&$2.1$e$+5$&$223$&$1.5$e$+5$&$120$&$3.6$e$+4$\\
      \hline
    \end{tabular}
    \caption{Comparison of  the number of iterations (\#Its) and  the solving time in seconds (s) for
      {\bf Target~1}.
    }
    \label{tab:costexample1}
  \end{table}

  \begin{table}[h]
    \begin{tabular}{|l|rl|rl|rl|rl|}
      \hline
      \multirow{2}{*}{Level}&\multicolumn{2}{c|}{${\mathcal{P}_{\text{mg}}\mbox{MINRES}}$} &
      \multicolumn{2}{c|}{${\mathcal{P}_{\text{diag}}\mbox{MINRES}}$} &
      \multicolumn{2}{c|}{BP-PCG} &
      \multicolumn{2}{c|}{inexSCPCG} \\  \cline{2-9}
      &\#Its&Time (s) &\#Its&Time (s) &\#Its&Time (s) &\#Its&Time (s)   \\
      \hline
      $L_1$ &$19$&$2.8$e$-2$&$21$&$2.9$e$-3$ &$24$&$2.5$e$-3$ &$10$&$6.4$e$-4$ \\
      $L_2$ &$23$&$2.9$e$-1$&$185$&$1.5$e$-1$&$184$&$3.1$e$-1$&$94$&$2.7$e$-2$ \\
      $L_3$ &$23$&$2.0$e$-0$&$258$&$1.5$e$-0$&$265$&$3.0$e$-0$&$133$&$2.8$e$-1$\\
      $L_4$ &$23$&$1.6$e$+1$&$256$&$5.2$e$+1$&$275$&$4.1$e$+1$&$138$&$3.5$e$-0$\\
      $L_5$ &$24$&$1.3$e$+2$&$248$&$1.1$e$+2$&$257$&$2.4$e$+2$&$137$&$3.9$e$+1$\\
      $L_6$ &$24$&$1.2$e$+3$&$240$&$1.9$e$+3$&$249$&$2.5$e$+3$&$134$&$6.4$e$+2$\\
      $L_7$ &$22$&$1.0$e$+4$&$230$&$2.9$e$+4$&$241$&$1.3$e$+4$&$129$&$3.3$e$+3$\\
      $L_8$ &$20$&$7.9$e$+4$&$220$&$2.2$e$+5$&$234$&$1.7$e$+5$&$123$&$2.6$e$+4$\\
      \hline
    \end{tabular}
    \caption{Comparison of  the number of iterations (\#Its) and  the  solving time in seconds (s) for
      {\bf Target~2}.
    }
             \label{tab:costexample2}
  \end{table}

   \begin{table}[h]
  \begin{tabular}{|l|rl|rl|rl|rl|}
       \hline
       \multirow{2}{*}{Level}&\multicolumn{2}{c|}{${\mathcal{P}_{\text{mg}}\mbox{MINRES}}$} &
       \multicolumn{2}{c|}{${\mathcal{P}_{\text{diag}}\mbox{MINRES}}$} &
      \multicolumn{2}{c|}{BP-PCG} &
      \multicolumn{2}{c|}{inexSCPCG} \\  \cline{2-9}
      &\#Its&Time (s) &\#Its&Time (s) &\#Its&Time (s) &\#Its&Time (s)  \\
      \hline
      $L_1$ &$21$&$3.2$e$-2$&$21$&$7.4$e$-2$ &$25$ &$2.5$e$-3$&$10$ &$5.9$e$-4$\\
      $L_2$ &$25$&$3.1$e$-1$&$191$&$1.5$e$-1$&$183$&$1.8$e$-1$&$97$ &$2.8$e$-2$\\
      $L_3$ &$25$&$2.2$e$-0$&$268$&$1.6$e$-0$&$272$&$3.4$e$-0$&$136$&$3.0$e$-1$\\
      $L_4$ &$25$&$1.8$e$+1$&$276$&$1.8$e$+1$&$285$&$3.1$e$+1$&$149$&$4.3$e$-0$\\
      $L_5$ &$26$&$1.5$e$+2$&$274$&$1.8$e$+2$&$284$&$1.1$e$+2$&$149$&$5.0$e$+1$\\
      $L_6$ &$26$&$1.5$e$+3$&$276$&$3.7$e$+3$&$279$&$2.6$e$+3$&$149$&$7.0$e$+2$\\
      $L_7$ &$26$&$1.2$e$+4$&$274$&$3.3$e$+4$&$266$&$1.7$e$+4$&$145$&$3.7$e$+3$\\
      $L_8$ &$26$&$1.5$e$+5$&$271$&$2.4$e$+5$&$237$&$1.6$e$+5$&$141$&$4.2$e$+4$\\
      \hline
    \end{tabular}
  \caption{Comparison of the  number of iterations (\#Its) and  the solving time in seconds (s) for
    {\bf Target~3}.
  }
    \label{tab:costexample3}
   \end{table}

   \begin{table}[h]
  \begin{tabular}{|l|rl|rl|rl|rl|}
      \hline
      \multirow{2}{*}{Level}&\multicolumn{2}{c|}{${\mathcal{P}_{\text{mg}}\mbox{MINRES}}$} &
      \multicolumn{2}{c|}{${\mathcal{P}_{\text{diag}}\mbox{MINRES}}$} &
      \multicolumn{2}{c|}{BP-PCG} &
      \multicolumn{2}{c|}{inexSCPCG} \\  \cline{2-9}
      &\#Its&Time (s) &\#Its&Time (s) &\#Its&Time (s) &\#Its&Time (s)   \\
      \hline
      $L_1$ &$21$&$3.5$e$-2$ &$21$ &$2.9$e$-3$ &$24$ & $2.1$e$-3$&$10$&$6.1$e$-4$ \\
      $L_2$ &$23$&$3.2$e$-1$ &$182$&$1.5$e$-1$ &$185$&$9.7$e$-2$ &$96$&$2.8$e$-2$ \\
      $L_3$ &$25$&$2.3$e$-0$ &$264$&$1.6$e$-0$ &$269$&$3.2$e$-0$ &$137$&$2.8$e$-1$ \\
      $L_4$ &$24$&$1.7$e$+1$ &$270$&$1.4$e$+1$ &$268$&$3.4$e$+1$ &$147$&$2.5$e$-0$ \\
      $L_5$ &$26$&$1.7$e$+2$ &$268$&$3.3$e$+2$ &$269$& $1.6$e$+2$&$148$& $2.3$e$+1$ \\
      $L_6$ &$26$&$1.6$e$+3$ &$271$&$4.2$e$+3$ &$267$&$4.2$e$+3$ &$150$&$4.9$e$+2$   \\
      $L_7$ &$26$&$1.3$e$+4$ &$268$&$4.0$e$+4$ &$266$&$2.4$e$+4$ &$149$&$2.8$e$+3$  \\
      $L_8$ &$24$&1.1$$e$+5$ &$265$&$2.7$e$+5$&$263$&$2.4$e$+5$ &$147$&$3.6$e$+4$ \\
      \hline
    \end{tabular}
  \caption{Comparison of the  number of iterations (\#Its) and the solving time in seconds (s) for
    {\bf Target~4}.
  }
      \label{tab:costexample4}
  \end{table}
   
%
%

\section{Conclusions and outlook}
\label{sec:ConclusionsOutlook}
We have derived robust estimates of the derivation of the finite element
approximation $u_{\varrho h}$ of the state $u_\varrho$ from the target
(desired state) $\overline{u}$ in the $L^2(\Omega)$ norm, and robust,
asymptotically optimal solvers for distributed elliptic optimal control
problems with $L^2$ -regularization. Due to the optimal choice
$\varrho = h^4$ of the regularization parameter, Jacobi-like
preconditioners are sufficient to construct MINRES or Bramble-Pasciak CG 
solvers of asymptotically optimal complexity with respect to arithmetical
operations  and memory demand. The parallelization of these iterative
methods is straightforward, and will lead to very scalable implementations
since, in contrast to multigrid preconditioners, diagonal preconditioners
are trivial to parallelize. The numerical results yield that the
multigrid preconditioned MINRES solver is slightly more efficient in a
single processor implementation.

Our numerical experiments  show that the inexact Schur Complement PCG
(inexactSCPCG) seems to be the most promising iterative solver, in
particular, in its parallel version, but also the single processor
implementation is the most efficient one in comparison with the MINRES
and Bramble-Pasciak CG  solvers. The numerical experiments also show that
the use of the inexact Schur complement, where the inverse of the mass
matrix is replaced by the inverse of the lumped mass matrix, does not
affect the accuracy. Here a rigoros numerical analysis is still needed. 
Moreover, the development of a nested iteration framework with an
a posteriori control of the discretization error and its parallel 
implementation is a future research topic.
An adaptive mesh refinement will probably require variable regularization 
functions $\varrho(x)$ adapted to the mesh density function rather than 
a fixed choice as we did in this paper where we have investigated 
uniform mesh refinement.

This approach is not only restricted to the simple model problem
  of the Poisson equation as constraint, extensions to more complicated
  elliptic equations, but also to parabolic, e.g., the heat equation,
  and hyperbolic, e.g., the wave equation, can be done in a similar
  way, and will be reported elsewhere. Moreover, the consideration of
  control constraints requires the efficient solution of a sequence
  of linear algebraic systems as considered in this paper.

%
%

\section*{Acknowledgments}
The authors would like to acknowledge the computing support of the
supercomputer MACH--2\footnote{https://www3.risc.jku.at/projects/mach2/}
from Johannes Kepler Universit\"{a}t Linz and of the high performance
computing cluster Radon1\footnote{https://www.oeaw.ac.at/ricam/hpc}
from Johann Radon Institute for Computational and Applied Mathematics
(RICAM) on which the numerical examples are performed. 


\bibliography{LangerLoescherSteinbachYang}
\bibliographystyle{abbrv}


\end{document}